\documentclass[11pt,reqno]{amsart}
\usepackage{enumerate}
\usepackage{amsrefs}
\usepackage{amsfonts, amsmath, amssymb, amscd, amsthm, bm, cancel,mathrsfs, upgreek}
\usepackage{url}
\usepackage{graphicx}
\usepackage[
linktocpage=true,colorlinks,citecolor=magenta,linkcolor=blue,urlcolor=magenta]{hyperref}
\usepackage{multicol}
\usepackage{comment}
\usepackage{tikz}
\usepackage[margin=1in]{geometry}

\newtheorem{thm}{Theorem}[section]

\newtheorem*{thmA'}{Theorem A'}

\newtheorem{cor}[thm]{Corollary}
\newtheorem{con}[thm]{Conjecture}
\newtheorem{lem}[thm]{Lemma}

\newtheorem{prop}[thm]{Proposition}

\theoremstyle{definition}
\newtheorem{defn}[thm]{Definition}
\newtheorem*{ack}{Acknowledgments}
\theoremstyle{remark}
\newtheorem{rem}{Remark}

\numberwithin{equation}{section}
\allowdisplaybreaks

\renewcommand{\widetilde}{\tilde}

\begin{document}
	\title[A nonlocal isoperimetric problem on $\mathbb{H}^n$]{Existence and nonexistence results for a nonlocal isoperimetric problem on $\mathbb{H}^n$}	
\author[H. Li]{Haizhong Li}
\address{Department of Mathematical Sciences, Tsinghua University, Beijing 100084, P.R. China}
\email{\href{mailto:lihz@tsinghua.edu.cn}{lihz@tsinghua.edu.cn}}
\author[B. Yang]{Bo Yang}
\address{Department of Mathematical Sciences, Tsinghua University, Beijing 100084, P.R. China}
\email{\href{mailto:ybo@tsinghua.edu.cn}{ybo@tsinghua.edu.cn}}
\address{Institut f\"{u}r Diskrete Mathematik und Geometrie, Technische Universit\"{a}t Wien, Wiedner Hauptstra{\ss}e 8-10, 1040 Wien, Austria}
\email{\href{mailto:bo.yang@tuwien.ac.at}{bo.yang@tuwien.ac.at}}
	

	\begin{abstract}
		In Euclidean space $\mathbb{R}^n$, the minimization problem of a nonlocal isoperimetric functional with a competition between perimeter and a nonlocal term derived from the negative power of the distance function, has been extensively studied. In this paper, we investigate this nonlocal isoperimetric problem in hyperbolic space $\mathbb{H}^n$, we prove that the geodesic balls are unique minimizers (up to hyperbolic isometries) for small volumes $m$ and obtain nonexistence results for large volumes $m$ under certain ranges of the exponent in the nonlocal term.
	\end{abstract}	
	\maketitle
		\tableofcontents
\section{Introduction}
Let $n\geq 2$, $0<\alpha<n$ and $\gamma>0$. For any measurable set $F\subset\mathbb{H}^n$, we consider the following functional:
\begin{equation}\label{defn-Functional}
	\mathcal{E}(F)=P(F)+\gamma NL_{\alpha}(F),
\end{equation}
where
\begin{equation}\label{defn-Perimeter}
 P(F)=\sup\left\{\int_{F}\mathrm{div}_g X\,dV_g:X\in C_c^1(\mathbb{H}^n,\mathrm{T}\mathbb{H}^n),|X|_g\leq 1\right\},
 \end{equation}
 denotes the perimeter of $F$ in the sense of  De Giorgi and the second term $NL_{\alpha}(F)$ is defined as
 \begin{equation}\label{defn-doubleint}
 	NL_{\alpha}(F)=\int_{\mathbb{H}^n}\int_{\mathbb{H}^n}{\frac{\chi_F(x)\chi_F(y)}{{d_g(x,y)}^{\alpha}}}\,dV_g(x)dV_g(y),
 \end{equation}
 where $\chi_F$ denotes the characteristic function of $F$. Furthermore,  in both \eqref{defn-Perimeter} and \eqref{defn-doubleint}, $g$ denotes the metric of $\mathbb{H}^n$ and in this paper, all quantities with subscript $g$ indicate that they are considered in $\mathbb{H}^n$. In the following, we investigate the minimization problem
 \begin{equation}\label{Prob-mini}
 	E(m)=\inf_{|F|_g=m}\mathcal{E}(F).
 \end{equation}
The minimization problem \eqref{Prob-mini} of functional \eqref{defn-Functional} has a physical background in Euclidean space $\mathbb{R}^n$ and has been extensively studied. The case where $\alpha=1$ in $\mathbb{R}^3$ is particularly significant, recalling Gamow's liquid drop model for atomic nuclei \cite{Gamow1930}. In this model, the nucleus is viewed as a region $\Omega\subset\mathbb{R}^3$ densely packed with nucleons (protons and neutrons) at constant density, implying the nucleon count is proportional to the volume of $\Omega$. The perimeter term in the energy functional corresponds to the surface tension binding the nucleus and the second term represents the Coulomb repulsion between protons (here all physical constants have been normalized to 1 for simplicity). Then the minimization problem is to ask that whether stable atomic nuclei at arbitrarily high atomic numbers can exist in the framework of the classical liquid drop model of nuclear matter.

In the case of $\mathbb{R}^n$, the minimization problem of functional \eqref{defn-Functional} is highly nontrivial since by the isoperimetric inequality, the perimeter  term attains its minimum on balls for fixed volume, while the second term attains its maximum on balls by the Riesz's rearrangement inequality (cf. \cite[Theorem 3.7]{Leib-Loss01}). For set $F\subset\mathbb{R}^n$ measurable with volume $|F|=m$, we scale $F$ by $F\to m^{\frac{1}{n}}U$, then $|U|=1$ and using the scaling properties of $\mathbb{R}^n$, we have
\begin{equation}\label{Ex-scale}
	\mathcal{E}(F)=m^{\frac{n-1}{n}}\left(P(U)+\gamma m^{\frac{n+1-\alpha}{n}}NL_{\alpha}(U)\right),
\end{equation}
which suggests that the perimeter term is dominant if volume $m$ is small and the nonlocal term $NL_{\alpha}(\cdot)$ is dominant if volume $m$ is large. Hence, we expect that there exist minimizers for small $m$ and there exist no minimizers for large $m$. Indeed, we define the critical volume $m_{*}$ to be the volume such that the value of the functional $\mathcal{E}(\cdot)$ of the ball with volume $m_{*}$ equals that of two balls with both volume $m_{*}/2$, spaced infinitely far apart, i.e.,
\begin{equation*}
	\mathcal{E}\left(\left(\frac{m_{*}}{b_n}\right)^{\frac{1}{n}}B\right)=2\mathcal{E}\left(\left(\frac{m_{*}}{2 b_n}\right)^{\frac{1}{n}}B\right),
\end{equation*}
where $B$ denotes the unit ball in $\mathbb{R}^n$ and $b_n=|B|$. In the follwing of paper, we also denote $\omega_{n-1}:=|\partial B|$, which is the area of the unit sphere in $\mathbb{R}^n$. Note that by coarea formula, we have $\omega_{n-1}=n b_n$.  A direct calculation shows
\begin{equation}\label{Eq-Cri}
	m_{*}=\left(\frac{2^{\frac{1}{n}}-1}{1-2^{\frac{\alpha-n}{n}}}\frac{\omega_{n-1}}{\gamma NL_{\alpha}(B)}\right)^{\frac{n}{n+1-\alpha}}b_n.
\end{equation}
The significance of $m_{*}$ is that for $m>m_{*}$, the value of the functional $\mathcal{E}(\cdot)$ of two balls with both volume $m/2$, spaced infinitely far apart is less than that of a ball with volume $m$. Moreover, it is conjectured that
\begin{con}[cf. \cite{Choksi-Peletier11}]\label{Con-Cri}
	Given $n\geq 2$, $0<\alpha<n$ and $\gamma>0$, for $m\leq m_{*}$, the ball of volume $m$ uniquely minimizes $\mathcal{E}(\cdot)$ among measurable sets $F\in\mathbb{R}^n$ with $|F|=m$, and for $m>m_{*}$ no minimizer exists.
	\end{con}
So far, Conjecture \ref{Con-Cri} has not been completely solved. Kn$\ddot{\mathrm{u}}$pfer and Muratov \cites{Knupfer-Muratov13,Knupfer-Muratov14} proved the following results:
\begin{itemize}
	\item[(a)] For every $n\geq 2$, $\alpha\in(0,n)$ and $\gamma>0$, there exists a constant $m_{c_1}>0$ depending on $n,\alpha$ and $\gamma$ such that $E(m)$ has a minimizer for every $m\leq m_{c_1}$.
	\item[(b)] For every $n\geq 2$, $\alpha\in(0,2)$ and $\gamma>0$, there exists a constant $m_{c_2}>0$ depending on $n,\alpha$ and $\gamma$ such that $E(m)$ has no minimizer for every $m>m_{c_2}$.
	\item[(c)] For fixed $\gamma>0$, if $n=2$ and $\alpha>0$ is suffciently small, then $m_{c_1}=m_{c_2}=m_{*}$, and balls are unique minimizers for $E(m)$ when $m\leq m_{*}$.
	\item[(d)] For fixed $\gamma>0$, if $n=2$ and $\alpha<2$, or if $3\leq n\leq 7$ and $\alpha<n-1$, there exists $0<m_{c_0}\leq m_{c_1}$, depending on $n,\alpha$ and $\gamma$ such that balls are unique minimizers for $E(m)$ with $m<m_{c_0}$.
\end{itemize}

Since then, much progress has been made on the basis of these results. Julin \cite{Julin14} extended the result of (d) to all $n\geq 3$ and $\alpha=n-2$. Bonacini and Cristoferi \cite{Bonacini-Cristoferi14} extended the results of (c) and (d) to all $n\geq 2$. Figalli, Fusco, Maggi, Millot and Morini \cite{Figalli15} extended the result of (d) to to all $n\geq 2$ and $\alpha\in(0,n)$. Frank and Nam \cite{Frank-Nam21} extended the result of (a) to $m_{c_1}=m_{*}$, they also showed that (b) holds for $0<\alpha<n$ and $\alpha\leq 2$. There are also some quantitative estimates for $\alpha$, $m_{c_0}$ and $m_{c_2}$. In the case of $n=2$, Muratov and Zaleski \cite{Muratov15} showed that the result of (c) holds for $\alpha\leq 0.034$.  In the case of $n=3$ and $\alpha=1$, Frank, Killip and Nam \cite{Frank-Killp-Nam16} proved that $m_{c_2}\leq 2.38m_{*}$, while Chodosh and Ruohoniemi \cite{Chodosh-Ruohomiemi25} showed that $m_{c_0}\geq 0.28m_{*}$. We also refer readers to \cites{Alama21,Carazzato23,Carazzato25,Choksi17,Choksi22,Frank19,Frank-Lieb15,Goldman15,Goldman25,Muratov18,Novaga21,Novaga22,Pascale25} for some other interesting results on the nonlocal isoperimetric problem.

Now we turn back to the case of $\mathbb{H}^n$. In this situation, the perimeter term also attains its minimum on geodesic balls with fixed volume by the isoperimetric inequality, and the second term attains its maximum on geodesic balls with fixed volume by the Symmetrization Lemma (cf. \cite[Page 225]{Beckner93}). Hence the minimization problem \eqref{Prob-mini} in $\mathbb{H}^n$ is also highly nontrivial and we expect that there are similar results in $\mathbb{H}^n$ compared with $\mathbb{R}^n$.

Before stating our main results, we make some measure-theoretic notations. We say that a function $u\in L^1(\mathbb{H}^n)$ has bounded variation, i.e. $u\in BV(\mathbb{H}^n)$, if there holds
\begin{equation*}
	\int_{\mathbb{H}^n}{|\nabla u|_g}\,dV_g:=\sup\left\{\int_{\mathbb{H}^n}u\,\mathrm{div}_g X\,dV_g: X\in C_c^1(\mathbb{H}^n,\mathrm{T}\mathbb{H}^n), |X|_g\leq 1\right\}<\infty.
\end{equation*}
For any measurable set $F\subset\mathbb{H}^n$, we say that $F$ has finite perimeter if its characteristic function $\chi_F\in BV(\mathbb{H}^n)$ and then its perimeter can be denoted as $P(F):=\int_{\mathbb{H}^n}{|\nabla\chi_F|}\,dV_g$. The $k$-dimensional Hausdorff measure with $k\in[0,n]$ is denoted as $H^k(F)$.

For any $H^n$-measurable set $F$, its upper and lower densities at a point $x\in\mathbb{H}^n$ are defined respectively by 
\begin{equation}\label{Defn-uppden}
	\overline{D}(F,x):=\limsup_{r\to 0}\frac{|F\cap B_r(x)|_g}{|B_r(x)|_g},\quad  \underline{D}(F,x):=\liminf_{r\to 0}\frac{|F\cap B_r(x)|_g}{|B_r(x)|_g},
\end{equation}
where $B_r(x)\subset\mathbb{H}^n$ is the open geodesic ball with center $x\in\mathbb{H}^n$ and radius $r$. If the upper and lower densities are equal, their common value will be called the density of $F$ at $x$, and it will be denoted by $D(F,x)$. 

The essential interior $\mathring{F}^M$ of $F$ is then defined as the set of all $x\in\mathbb{H}^n$ for which $D(F,x)=1$, while the essential closure $\bar{F}^{M}$ of $F$ is defined as the set of all points where $\overline{D}(F,x)>0$. The essential boundary $\partial^M F$ of $F$ is defined as the set of points where $\overline{D}(F,x)>0$ and $\overline{D}(\mathbb{H}^n\setminus F,x)>0$.  For any two sets $A, B\subset\mathbb{H}^n$, we denote $A\Delta B:=(A\setminus B)\cup (B\setminus A)$ as the symmetric difference of $A$ and $B$, then by the Lebesgue–Besicovitch Differentiation Theorem (cf. \cite[Theorem 1.32]{Evans15}), we have
\begin{equation*}
	|F\Delta\mathring{F}^M|_g=|F\Delta\bar{F}^{M}|_g=0.
\end{equation*}
We also have that (cf. \cite[Page 46]{Ambrosio01})
\begin{equation}\label{Eq-Essenb}
	\partial^M F=\mathbb{H}^n\setminus(\mathring{F}^M\cup(\mathring{F}^c)^M).
\end{equation}

The reduced boundary $\partial^{*}F$ of a set of finite perimeter F is defined as a set of points $x\in\partial^M F$ such that the measure-theoretic normal exists at $x$, i.e., if the following limit exists:
\begin{equation*}
	\nu_F(x):=\lim_{r\to 0}\frac{\int_{B_r(x)}{\nabla{\chi_{F}}\,dV_g}}{\int_{B_r(x)}{|\nabla{\chi_{F}}|_g\,dV_g}}\quad \text{and}\quad |\nu_F(x)|_g=1,
\end{equation*}
where $\nabla\chi_F$ is the vector-valued Radon measure associated with the distributional derivative of $\chi_F$ and $|\nabla{\chi_{F}}|_g$ coincides with the $H^{n-1}$ measure restricted to $\partial^M F$. If $F$ is a set of finite perimeter, then by De Giorgi’s structure theorem (cf. \cite[Theorem 3.59]{Ambrosio-Fusco-Pallara00}, \cite[Theorem 15.9]{Maggi-12}), we have $P(F)=H^{n-1}(\partial^{*}F)$. Also, by a result of Federer (cf. \cite[Theorem 3.61]{Ambrosio-Fusco-Pallara00}), we have $\partial^{*} F\subset\partial^M F$ and $H^{n-1}(\partial^M F\setminus\partial^{*} F)$=0. 
\begin{rem}
	For any set $A\subset\mathbb{H}^n$, we can define the relative perimeter of $F$ in $A$ as $P(F;A):=|\nabla{\chi_{F}}|(A)=H^{n-1}(\partial^{*}F\cap A)$. We say that $F$ is a set of locally finite perimeter, if for every compact set $K\subset\mathbb{H}^n$, we have $P(F;K)<\infty$.
\end{rem}

We say that a set $F$ of finite perimeter is essentially bounded, if its essential closure $\bar{F}^{M}$ is bounded. Also, we say $F$ is decomposable, if there exists a partition $(A,B)$ of $F$ with both $|A|_g>0$ and $|B|_g>0$, such that $P(F)=P(A)+P(B)$. Otherwise, we say that $F$ is indecomposable.
\begin{rem}
	All the measure-theoretic notations mentioned above suit the case of $\mathbb{R}^n$ if the corresponding quantities are considered in $\mathbb{R}^n$.
\end{rem}

Our first result says that the functional $\mathcal{E}(\cdot)$ admits geodesic ball as the unique minimizer (up to hyperbolic isometries) when the volume $m$ is small.
\begin{thm}\label{Thm-Existence}
	For all $n\geq 2$, $0<\alpha<n$ and $\gamma>0$, there exists a positive constant $m_1$ depending on $n,\alpha$ and $\gamma$ such that for $0<m\leq m_1$, the unique minimizer (up to hyberbolic isometries) to the minimization problem \eqref{Prob-mini} with volume $m$ is given by a geodesic ball. 
\end{thm}
\begin{rem}
	It remains an open problem to provide  explicit lower bounds on $m_1$.
\end{rem}
\begin{rem}
	Through a more refined analysis of the constants appearing in the estimates, we can further show that if $0<\alpha\leq\bar{\alpha}$ and $0<\gamma\leq\bar{\gamma}$ for two fixed $\bar{\alpha}\in(0,n)$ and $\bar{\gamma}>0$, then the positive constant $m_1$ in Theorem \ref{Thm-Existence} can be chosen to depend only on $n$, $\bar{\alpha}$ and $\bar{\gamma}$.
\end{rem}
The proof of this theorem in $\mathbb{H}^n$ differs significantly from that in $\mathbb{R}^n$. In $\mathbb{R}^n$, scaling transformations possess particularly convenient geometric properties. In particular, they allow one to rescale the volume to an arbitrary prescribed value while preserving the shape of the underlying set, and geodesic balls are mapped to geodesic balls under such transformations. Consequently, the analysis of the shape of minimizers in the small-volume regime can be reduced to the study of the corresponding minimization problem for a fixed reference volume $b_n$, as presented in \cites{Bonacini-Cristoferi14,Knupfer-Muratov13,Knupfer-Muratov14,Figalli15}. 

However, such transformations are lacking in $\mathbb{H}^n$. Fortunately, after comparing several common models of $\mathbb{H}^n$, we have adopted the $\Phi_{\lambda}$-transformation (see \S\ref{Subsec-model} for details) in the upper half-space model of $\mathbb{H}^n$, which has the advantage that by adjusting the value of parameter $\lambda$, the volume of any set in $\mathbb{H}^n$ can be scaled to a given multiple. By means of the $\Phi_{\lambda}$-transformation, we can establish several fundamental properties of the minimizers, including their regularity, essential boundedness and indecomposability, and we can further deduce the existence of minimizers for all sufficiently small volumes $m\leq m_0$ for some positive $m_0$ depending only on $n,\alpha$ and $\gamma$. 

When proving that geodesic balls are the unique minimizers for small volumes, we encounter a new difficulty. Since our $\Phi_{\lambda}$-transformation cannot map geodesic balls to geodesic balls, the argument for a fixed volume of $b_n$ in $\mathbb{R}^n$ is inapplicable. A key observation is that if the minimizer $E$ satisfies $|E|_g=|B_r|_g=m$ for $m\leq m_0$ and some positive $r$, then $E$ is contained in a geodesic ball of radius $(1+C_3 r^{\frac{1}{2n}})r$ for some uniform constant $C_3$. Combining this observation with a notion of quasiminimizer of the perimeter, we are able to show that $\partial E$ can be viewed as a radial graph over $\mathbb{S}^{n-1}$ with the radial function close to $r$ provided the volume $m$ is small enough. Then by means of proof by contradiction and Fuglede's type estiamtes for the perimeter and the nonlocal term, we are able to prove that for small volumes $m$, geodesic balls are unique minimizers up to hyperbolic isometries.
		
On the other hand, by adopting similar ideas from \cite{Knupfer-Muratov14} and leveraging the $\Phi_{\lambda}$-transformation, we can prove that the functional $\mathcal{E}(\cdot)$ does not admit a minimizer for large volume $m$ with an extra assumption that $\alpha<2$.
\begin{thm}\label{Thm-Nonexistence}
	For all $n\geq 2$, $0<\alpha<2$ and $\gamma>0$, there exists a positive constant $m_2$ depending on $n,\alpha$ and $\gamma$ such that for $m\geq m_2$, the minimization problem \eqref{Prob-mini} does not admit a minimizer with volume $m$.
\end{thm}
\begin{rem}
	We find it difficult to extend Theorem \ref{Thm-Nonexistence} to the case of $\alpha=2$, as in the work of Frank and Nam \cite{Frank-Nam21}, the reason being that we are unsure how to obtain the estimate in \cite[Lemma 7]{Frank-Nam21}.
\end{rem}
\begin{rem}
	Motivated by the classical results in $\mathbb{R}^n$ (cf. \cite[Theorem 2.9]{Bonacini-Cristoferi14} and \cite[Theorem 1.5]{Figalli15}), we intend to calculate the first and second variations of the functional \eqref{defn-Functional} in our subsequent work, so as to provide explicit bounds for the volume within which geodesic balls are local minimizers. As an application, we aim to derive similar quantitative results regarding parameter $m_1$ as those presented in \cite{Chodosh-Ruohomiemi25}.
\end{rem}
The paper is organized as follows: In \S\ref{Sec-pre}, we collect some preliminaries which will be used later in this paper, including the upper-half space model of $\mathbb{H}^n$, some isperimetric-type inequalities in $\mathbb{H}^n$, two notions of quasiminimizer of the perimeter, and Fuglede’s type estimates of $P(E)$ and $NL_{\alpha}(E)$ for a nearly spherical set $E\subset\mathbb{H}^n$. In \S\ref{Sec-Basic es}, we prove some basic properties for the minimizer of the functional $\mathcal{E}(\cdot)$, including the regularity result, the boundedness, the indecomposability and then prove a uniform lower density bound. During this process, we also derive a general criterion on a set of finite perimeter being not a minimizer. In \S\ref{Sec-Existence}, we prove that there exists a minimizer if the volume $m$ is small and it is contained in a geodesic ball with small radius. In \S\ref{Sec-Ball}, we prove that the functional $\mathcal{E}(\cdot)$ admits geodesic ball as unique minimizers (up to hyperbolic isometries) when volumes $m$ are small. In \S\ref{Sec-Nonex}, we firstly prove that both the the minimization value \eqref{Prob-mini} and the diameter of the essential closure of a minimizer can be two-sided controlled by the volume $m$, and then complete the proof that there exist no minimizers for large volumes.

\begin{ack}
	The research was supported by NSFC Grant No.12471047 and  China Postdoctoral Science Foundation No.2024M751605.
\end{ack}
\section{Preliminaries}\label{Sec-pre}
\subsection{Upper half-space model of $\mathbb{H}^n$}\label{Subsec-model} The $n$-dimensional hyperbolic space $\mathbb{H}^n$ can be viewed as $\mathbb{H}^n=(\mathbb{R}^n_{+},g)$, where $\mathbb{R}^n_{+}=\{(x_1,x_2,\dots,x_n)\in\mathbb{R}^n:x_n>0\}$ is the upper half-space and the metric $g=\frac{1}{x_n^2}\delta$. Here $\delta$ denotes the canonical metric of $\mathbb{R}^n$. Usually, we denote this model as $U^n$, and call it the upper half-space model of $\mathbb{H}^n$.
\begin{prop}{\cite[Theorem 4.6.1]{Ratcliffe19}}
	For any two points $x=(x_1,\dots,x_n), y=(y_1,\dots,y_n)\in U^n$, the metric $d_g$ on $U^n$ satisfies
	\begin{equation}\label{Ex-metric}
		\cosh d_g(x,y)=1+\frac{|x-y|^2}{2x_n y_n},
	\end{equation}
where $|x-y|$ denotes the Euclidean distance between $x$ and $y$.
\end{prop}
For any fixed $\lambda>0$, we define the transformation $\Phi_{\lambda}$ which maps $x=(x_1,\dots,x_n)$ to $x':=\Phi_{\lambda}(x)=(x_1',\dots,x_n')$ as follows:
\begin{align}
	\Phi_{\lambda}: U^n&\to U^n,\notag\\
	(x_1,\dots,x_{n-1},x_n)&\mapsto(x_1,\dots,x_{n-1},\lambda x_n).\label{Defn-Phi}
\end{align}
Obviously, $\Phi_{\lambda}$ is bijective with $\Phi_{\lambda}^{-1}=\Phi_{{\lambda}^{-1}}$.  In the rest of this subsection, we show some basic estimates about the transformation $\Phi_{\lambda}$ which will be used in this paper.
\begin{prop}\label{Prop-Lip}
	Given fixed $\lambda>0$ and $x,y\in U^n$, there holds
	\begin{equation}\label{Es-dis}
		\min\{{\lambda}^{-2},1\}d_g(x,y)\leq d_g(\Phi_{\lambda}(x),\Phi_{\lambda}(y))\leq \max\{{\lambda}^{-2},1\} d_g(x,y).
	\end{equation}
\end{prop}
\begin{proof}
	Denote $x':=\Phi_{\lambda}(x)=(x_1',\dots,x_n')$ and $y':=\Phi_{\lambda}(y)=(y_1',\dots,y_n')$. Then we have that $x_i'=x_i$ for $1\leq i\leq n-1$ and $x_n'=\lambda x_n$, so does $y'$. Hence by \eqref{Ex-metric}, there holds
	\begin{equation}\label{Eq-dis1}
		\cosh d_{g}(\Phi_{\lambda}(x),\Phi_{\lambda}(y))=\cosh d_{g}(x',y')=1+\frac{\sum_{i=1}^{n-1}(x_i-y_i)^2+{\lambda}^2(x_n-y_n)^2}{2{\lambda}^2 x_n y_n}.
			\end{equation}
$\bullet\,\,\bf{Case\,1.}$ If $\lambda\geq 1$, then by \eqref{Eq-dis1}, we obviously have
	\begin{equation}\label{Eq-dis2}
		\cosh d_{g}(\Phi_{\lambda}(x),\Phi_{\lambda}(y))\leq 1+\frac{{\lambda}^2\sum_{i=1}^{n-1}(x_i-y_i)^2+{\lambda}^2(x_n-y_n)^2}{2{\lambda}^2 x_n y_n}=\cosh d_g(x,y).
	\end{equation}
On the other hand, 
\begin{align}
	\cosh d_{g}(\Phi_{\lambda}(x),\Phi_{\lambda}(y))&\geq 1+\frac{|x-y|^2}{2{\lambda}^2 x_n y_n}\notag\\
	&=\frac{1}{{\lambda}^2}\cosh d_g(x,y)+1-\frac{1}{{\lambda}^2}\notag\\
	&\geq \cosh\left(\frac{1}{{\lambda}^2}d_g(x,y)\right).\label{Eq-dis3}
\end{align}
Indeed, if we let $a=\frac{1}{{\lambda}^2}\in(0,1]$ and consider the function $f(t)=a\cosh t+1-a-\cosh(at)$ for $t\geq 0$, then since $f(0)=0$ and $f'(t)=a(\sinh t-\sinh(at))\geq 0$, we get \eqref{Eq-dis3}. Hence combining \eqref{Eq-dis2} with \eqref{Eq-dis3} and using the fact that $g(t)=\cosh t$ is an increasing function in $t$ gives \eqref{Es-dis} in the case of $\lambda\geq 1$. 

\noindent$\bullet\,\,\bf{Case\,2.}$ If $0<\lambda<1$, we can argue as {\bf{Case 1}} to get the estimate
\begin{equation*}
		d_g(x,y)\leq d_g(\Phi_{\lambda}(x),\Phi_{\lambda}(y))\leq\frac{1}{{\lambda}^2}d_g(x,y).
\end{equation*}
Combining the above two cases gives the proof of this proposition. 
\end{proof}
Then we show that how the volume and perimeter of a measurable set change under the transformation $\Phi_{\lambda}$.
\begin{prop}\label{Prop-VP}
	For $F\subset\mathbb{H}^n$ measurable and of finite perimeter, we have
	\begin{align}
		|\Phi_{\lambda}(F)|_g&={\lambda}^{1-n}|F|_g,\label{Re-volume}\\
		\min\{{\lambda}^{2-2n},1\}P(F)&\leq P(\Phi_{\lambda}(F))\leq \max\{{\lambda}^{2-2n},1\}P(F).\label{Re-perimeter}
	\end{align}
\end{prop}
\begin{proof}
	Given arbitrary $x\in U^n$, since the metric $g$ satisfies $g=\frac{1}{x_n^2}\delta$, where $\delta$ is the canonical metric of $\mathbb{R}^n$. Then the volume element of $U^n$ at $x$ satisfies
	\begin{equation}\label{Eq-volumeform}
		dV_g(x)=\frac{1}{x_n^n}dx_1\cdots dx_n,
	\end{equation}
	where $dx:=dx_1\cdots dx_n$ is the volume element on $\mathbb{R}^n$. Hence using \eqref{Eq-volumeform}, we have
	\begin{align*}
		|\Phi_{\lambda}(F)|_g&=\int_{\Phi_{\lambda}(F)}{dV_g(x')}=\int_{\Phi_{\lambda}(F)}{\frac{1}{(x_n')^n}}\,dx_1'\cdots dx_n'\\
		&=\lambda^{1-n}\int_F{\frac{1}{x_n^n}}\,dx_1\cdots dx_n\\
		&=\lambda^{1-n}|F|_g,
	\end{align*}
which gives \eqref{Re-volume}.

To get \eqref{Re-perimeter}, we firstly prove that $\partial^M \Phi_{\lambda}(F)=\Phi_{\lambda}(\partial^M F)$. For a general set $E\subset U^n$, it is obviously that $\Phi_{\lambda}(E)\cap B_r(x)=\Phi_{\lambda}(E\cap\Phi_{\lambda}^{-1}(B_r(x))$ for any $x\in U^n$, which gives
\begin{equation}\label{Eq-P1}
	\frac{|\Phi_{\lambda}(E)\cap B_r(x)|_g}{|B_r(x)|_g}=\frac{|\Phi_{\lambda}(E\cap\Phi_{\lambda}^{-1}(B_r(x)))|_g}{|B_r(x)|_g}=\frac{|E\cap\Phi_{\lambda}^{-1}(B_r(x))|_g}{|\Phi_{\lambda}^{-1}(B_r(x))|_g}.
\end{equation} 
By \cite[Theorem 4.6.4; Exercise 4.6.6]{Ratcliffe19}, we see that the geodesic ball $B_r(x)$ in $U^n$ is actually an Euclidean ball contained in the upper half space $\mathbb{R}^n_{+}$ with center $x(r)=(x_1,\dots,x_{n-1},x_n\cosh r)$ and radius $x_n\sinh r$. Hence $\Phi_{\lambda}^{-1}(B_r(x))$ is indeed an Euclidean ellipsoid which shrinks to the point $\Phi_{\lambda}^{-1}(x)$ as $r$ tends to 0. Since the (upper) density of a point keeps invariant regardless of the shape of the sets used to shrink to the point, as long as they are bounded and have bounded eccentricity, we deduce from \eqref{Eq-P1} that the (upper) density of $\Phi_{\lambda}(E)$ at the point $x$ equals to that of $E$ at the point $\Phi_{\lambda}^{-1}(x)$. If we take $E=F$ and $E=U^n\setminus F$ respectively, and notice that $\Phi_{\lambda}(U^n\setminus F)=U^n\setminus{\Phi_{\lambda}(F)}$, then using \eqref{Defn-uppden} and \eqref{Eq-P1} gives $\partial^M \Phi_{\lambda}(F)=\Phi_{\lambda}(\partial^M F)$.

Finally, by Proposition \ref{Prop-Lip}, $\Phi_{\lambda}$ is a Lipschitz map with Lipschitz constant $\mathrm{Lip}\,\Phi_{\lambda}\leq \max\{{\lambda}^{-2},1\}$ and since $F$ is of finite perimeter, we have
\begin{align*}
	P(\Phi_{\lambda}(F))&=H^{n-1}(\partial^M \Phi_{\lambda}(F))=H^{n-1}(\Phi_{\lambda}(\partial^M F))\\
	&\leq (\mathrm{Lip}\,\Phi_{\lambda})^{n-1}H^{n-1}(\partial^M F)\\
	&\leq \max\{{\lambda}^{2-2n},1\}P(F).
\end{align*}
The left hand side of inequality \eqref{Re-perimeter} also holds since $F=\Phi_{{\lambda}^{-1}}(\Phi_{\lambda}(F))$ and $\Phi_{\lambda}^{-1}$ is a Lipschitz map of $U^n$ with Lipschitz constant $\mathrm{Lip}_{\lambda^{-1}}\leq \max\{\lambda^2,1\}$ by Proposition \ref{Prop-Lip}. This completes of the proof of inequality \eqref{Re-perimeter}.
\end{proof}

\subsection{Isoperimetric-type inequalities in $\mathbb{H}^n$} In this subsection, we give some results concerning the isoperimetric-type inequalities in $\mathbb{H}^n$, including the relative isoperimetric inequality, the quantitative isoperimetric inequality and an Euclidean-like isoperimetric inequality for measurable sets in $\mathbb{H}^n$. We begin with a simple property of concave functions which will be used in the proof of the relative isoperimetric inequality.
\begin{lem}\label{Lem-concave}
	Assume that $\phi(x):[0,+\infty)\to[0,+\infty)$ is a $C^2$, concave function. Then for any $a,b>0$, we have 
	\begin{equation}\label{In-Concave}
		\phi(a)+\phi(b)\geq \phi(a+b)+\phi(0).
	\end{equation}
Moreover, if $\phi$ is strictly concave, the inequality is strict.
\end{lem}
\begin{proof}
	Without loss of generality, we assume that $0<a\leq b$, then by Langrange's mean value theorem, there exist $t_1\in(0,a)$ and $t_2\in(b,a+b)$ such that 
	\begin{align}
		\phi(a)-\phi(0)&=\phi'(t_1)a,\label{Eq-con1}\\
		\phi(a+b)-\phi(b)&=\phi'(t_2)a.\label{Eq-con2}
	\end{align}
	Since $\phi$ is $C^2$ and concave, we have $\phi''\leq 0$ and hence $\phi'(t_1)\geq \phi'(t_2)$, since $t_1<t_2$. This gives \eqref{In-Concave} obviously.  Moreover, if $\phi$ is strictly concave, then $\phi''\leq 0$ and $\phi''$ is not identically 0 on any interval of $[0,+\infty)$, which implies that $\phi'(t_1)> \phi'(t_2)$ and hence the inequality in  \eqref{In-Concave} is strict.
\end{proof}
In the rest of this paper, we denote $B_r$ as the geodesic ball in $\mathbb{H}^n$ with radius $r$, if there is no need to mention where the center is. Let $\xi:[0,+\infty)\to[0,+\infty)$ be an increasing function (which we usually call the isoperimetric function) such that 
\begin{equation}\label{Defn-xi}
	P(B_r)=\xi(|B_r|_g).
\end{equation}
If we denote $g(r)=|B_r|_g$, then by the coarea formula, we have
\begin{equation}\label{Eq-Volume}
	g(r)=|B_r|_g=\int_0^r{P(B_s)}\,ds=\omega_{n-1}\int_{0}^r{\sinh^{n-1}s}\,ds,
\end{equation}
hence $g'(r)=\omega_{n-1}\sinh^{n-1}r$. If we view $|B_r|_g$ as a new variable $z$, since $g'>0$, then the function $g$ has an inverse $g^{-1}$ with $r=g^{-1}(z)$. Moreover, 
\begin{equation*}
	\frac{dg^{-1}(z)}{dz}=\frac{1}{\omega_{n-1}\sinh^{n-1}(g^{-1}(z))}
\end{equation*}
and 
\begin{equation*}
	\xi(z)=P(B_{g^{-1}(z)})=\omega_{n-1}\sinh^{n-1}(g^{-1}(z)).
\end{equation*}
Combining the above two equations gives
\begin{align}
	\xi'(z)&=(n-1)\omega_{n-1}\sinh^{n-2}(g^{-1}(z))\cosh(g^{-1}(z))\frac{dg^{-1}(z)}{dz}=(n-1)\coth(g^{-1}(z)),\label{Eq-xi0}\\
	\xi''(z)&=-\frac{n-1}{\sinh^2(g^{-1}(z))}\frac{dg^{-1}(z)}{dz}=-\frac{n-1}{\omega_{n-1}\sinh^{n+1}(g^{-1}(z))}<0,\label{Eq-xi1}
\end{align}
from which we have the folllowing proposition:
\begin{prop}\label{Prop-concave}
	The isoperimetric function $\xi$ defined in \eqref{Defn-xi} is an increasing, strictly concave function on $[0,+\infty)$.
\end{prop} 
\begin{rem}\label{Rem-Iso}
	For any measurable set $F\subset\mathbb{H}^n$, we have the isoperimetric inequality: $P(F)\geq\xi(|F|_g)$, and equality holds if and only if $F$ is a geodesic ball.
	 \end{rem}

In \cite[Proposition 12.37; Remark 12.38]{Maggi-12}, Maggi proved the following relative isoperimetric inequality in $\mathbb{R}^n$ on an Euclidean ball:
\begin{prop}[Local perimeter bound on volume]\label{Prop-RE}
	If $n\geq 2$, $t\in(0,1)$, $x\in\mathbb{R}^n$ and $r>0$, then there exists a positive constant $c(n,t)$ such that 
	\begin{equation}\label{In-relative1}
		P(F;B_r(x))\geq c(n,t)|F\cap B_r(x)|^{\frac{n-1}{n}},
	\end{equation}
	for every set of locally finite perimeter $F$ such that $|F\cap B_r(x)|\leq t|B_r(x)|$. Moreover, there exists a constant $c(n)$ such that 
	\begin{equation}\label{In-relative2}
			P(F;B_r(x))\geq c(n)\min\{|F\cap B_r(x)|, |B_r(x)\setminus F|\}^{\frac{n-1}{n}}.
	\end{equation}
\end{prop}

The two constants appeared in Proposition \ref{Prop-RE} do not depend on $r$, since the classical isoperimetric inequalitiy for sets in $\mathbb{R}^n$ has a nice explicit expression. However, the isoperimetric function $\xi$ defined by \eqref{Defn-xi} in $\mathbb{H}^n$ can not be expressed explicitly in general, hence we can only prove an analogue result in $\mathbb{H}^n$ with constant depending on the radius $r$ of the geodesic ball.

\begin{prop}\label{Prop-REH}
	If $n\geq 2$, $t\in(0,1)$, $x\in\mathbb{H}^n$ and $0<r\leq r_0$ for some $r_0>0$. Then there exists a positive constant $c(n,t,r_0)$ such that 
	\begin{equation}\label{Prop-REH1}
		P(F;B_r(x))\geq c(n,t,r_0)\xi(|F\cap B_r(x)|_g),
	\end{equation}
for every set of locally finite perimeter $F$ such that $|F\cap B_r(x)|_g\leq t|B_r(x)|_g$. Moreover, there exists a constant $c(n,r_0)$ such that 
\begin{equation}\label{Prop-REH2}
	P(F;B_r(x))\geq c(n,r_0)\xi(\min\{|F\cap B_r(x)|_g, |B_r(x)\setminus F|_g\}).
\end{equation}
\end{prop}
\begin{proof}
	Since this proof is independent of the choice of $x$, hence we omit it in the following. By \cite[Lemma 15.12]{Maggi-12}, we have for a.e. $s>0$, there holds
	\begin{align*}
		P(F\cap B_s)&=H^{n-1}(F\cap\partial B_s)+P(F;B_s),\\
		P(B_s\setminus F)&=H^{n-1}(\partial B_s\setminus F)+P(F;B_s).
	\end{align*}
	If we denote $p(s):=P(F;B_s)$ and $h(s)=|F\cap B_s|_g$ for simplicity, which are both continuous functions in $s$, by Remark \ref{Rem-Iso} we have
	\begin{align}
		\xi(h(s))\leq P(F\cap B_s)&=H^{n-1}(F\cap\partial B_s)+p(s),\label{eq-Reh1}\\
		\xi(|B_s|_g-h(s))\leq P(B_s\setminus F)&=H^{n-1}(\partial B_s\setminus F)+p(s).\label{eq-Reh2}
	\end{align}
	If we add \eqref{eq-Reh1} and \eqref{eq-Reh2} together, and notice that $H^{n-1}(F\cap\partial B_s)+H^{n-1}(\partial B_s\setminus F)=P(B_s)=\xi(|B_s|_g)$, then we have that
	\begin{equation}\label{eq-Reh3}
			\xi(h(s))+\xi(|B_s|_g-h(s))\leq 2p(s)+\xi(|B_s|_g)
	\end{equation}
holds for a.e. $s>0$. Since both sides of the inequality \eqref{eq-Reh3} are continous in $s$, so \eqref{eq-Reh3} actually holds for all $s>0$, and particularly it holds for $s=r$. If we set $h(r)=\beta|B_r|_g$, and hence $|B_r|_g-h(r)=(1-\beta)|B_r|_g$, then \eqref{eq-Reh3} now becomes (taking $s=r$)
\begin{equation}\label{eq-Reh4}
	2p(r)\geq \xi(\beta|B_r|_g)+\xi((1-\beta)|B_r|_g)-\xi(|B_r|_g),
\end{equation}
where $0<\beta\leq t<1$ and $0<r\leq r_0$. Define the function 
\begin{equation}\label{Defn-Psi}
	\Psi(\beta,r)=\frac{\xi(\beta|B_r|_g)+\xi((1-\beta)|B_r|_g)-\xi(|B_r|_g)}{\xi(\beta|B_r|_g)}=1-\frac{\xi(|B_r|_g)-\xi((1-\beta)|B_r|_g)}{\xi(\beta|B_r|_g)}.
\end{equation}
By Lemma \ref{Lem-concave} and Proposition \ref{Prop-concave}, we have $\Psi(\beta,r)>0$. We next show that 
\begin{itemize}
	\item[(i)] For any $r_{*}>0$, $\lim_{(\beta,r)\to(0,r_{*}),\beta>0,r>0}\Psi(\beta,r)=1$;
	\item[(ii)] For any $\beta_{*}>0$, $\lim_{(\beta,r)\to(\beta_{*},0),\beta>0,r>0}\Psi(\beta,r)=1-\beta_{*}^{-\frac{n-1}{n}}[1-(1-\beta_{*})^{\frac{n-1}{n}}]$, and if we denote the right hand side of the limit as $k(\beta_{*})$, then $\lim_{\beta_{*}\to 0}k(\beta_{*})=1$ and $k(\beta_{*})\geq 1-t^{-\frac{n-1}{n}}[1-(1-t)^{\frac{n-1}{n}}]>0$ for $\beta_{*}\in(0,t]$.
	\item[(iii)] $\lim_{(\beta,r)\to(0,0),\beta>0,r>0}\Psi(\beta,r)=1$.
\end{itemize}
If we can prove the above three properties, then the function $\overline{\Psi}(\beta,r)$ defined by 
\begin{equation}\label{p}
	\overline{\Psi}(\beta,r)=\left\{\begin{aligned}
		&\Psi(\beta,r),\beta>0,r>0,\\
		&k(\beta), \beta>0,r=0,\\
		&1,\beta=0,r\geq 0
	\end{aligned}\right.
\end{equation}
is continous for $0\leq\beta\leq t$ and $0\leq r\leq r_0$, which in turn yields the existence of a constant $c_0(n,t,r_0)$, such that $\overline{\Psi}(\beta,r)\geq c_0(n,t,r_0)$ and hence \eqref{Prop-REH1} holds with $c(n,t,r_0)=\frac{1}{2}c_0(n,t,r_0)$.

To prove (i), for any sequence $(\beta_i,r_i)\to(0,r_{*})$ as $i\to\infty$, we have $\beta_i|B_{r_i}|_g\to 0$ as $i\to\infty$, assume that $s_i$ satisfies $|B_{s_i}|_g=\beta_i|B_{r_i}|_g$, then $s_i\to 0$ and $|B_{s_i}|_g\sim b_n s_i^n$ and $P(B_{s_i})\sim nb_ns_i^{n-1}$. Moreover, $s_i\sim\left(\frac{\beta_i}{b_n}|B_{r_i}|_g\right)^{\frac{1}{n}}$ and hence 
\begin{equation}\label{Es-Reh5}
	\xi(\beta_i|B_{r_i}|_g)=P(B_{s_i})\sim nb_ns_i^{n-1}\sim nb_n^{\frac{1}{n}}\left(\beta_i|B_{r_i}|_g\right)^{\frac{n-1}{n}}.
\end{equation}
On the other hand, by the Taylor expansion, there holds
\begin{align}
	\xi((1-\beta_i)|B_{r_i}|_g)&=\xi(|B_{r_i}|_g)-\beta_i|B_{r_i}|_g\xi'(|B_{r_i}|_g)+O(\beta_i^2|B_{r_i}|_g^2)\notag\\
	&=\xi(|B_{r_i}|_g)-(n-1)\beta_i|B_{r_i}|_g\coth{r_i}+O(\beta_i^2|B_{r_i}|_g^2).\label{Es-Reh5.5}
\end{align}
Combining \eqref{Defn-Psi} with \eqref{Es-Reh5} and \eqref{Es-Reh5.5} gives $\lim_{i\to\infty}\Psi(\beta_i,r_i)=1$.

To prove (ii), for any sequence $(\beta_i,r_i)\to(\beta_{*},0)$, we see that $|B_{r_i}|_g\sim b_n r_i^n$ and $P(B_{r_i})\sim nb_nr_i^{n-1}$. Assume that $\beta_i|B_{r_i}|_g=|B_{s_i^1}|_g$ and $(1-\beta_i)|B_{r_i}|_g=|B_{s_i^2}|_g$, then as $i\to\infty$, we have $s_i^1\sim \beta_i^{\frac{1}{n}}r_i$ and $s_i^2\sim (1-\beta_i)^{\frac{1}{n}}r_i$. Hence 
\begin{align}
	\xi(\beta_i|B_{r_i}|_g)=P(B_{s_i^1})&\sim nb_n (s_i^1)^{n-1}\sim nb_n\beta_i^{\frac{n-1}{n}}r_i^{n-1},\label{Es-Reh6}\\
	\xi((1-\beta_i)|B_{r_i}|_g)=P(B_{s_i^2})&\sim nb_n (s_i^2)^{n-1}\sim nb_n(1-\beta_i)^{\frac{n-1}{n}}r_i^{n-1}.\label{Es-Reh7}
\end{align}
Combining \eqref{Defn-Psi} with \eqref{Es-Reh6} and \eqref{Es-Reh7} gives 
\begin{align}
	\lim_{i\to\infty}\Psi(\beta_i,r_i)&=1-\lim_{i\to\infty}\beta_{i}^{-\frac{n-1}{n}}[1-(1-\beta_{i})^{\frac{n-1}{n}}]\label{Eq-Reh8}\\
	&=1-\beta_{*}^{-\frac{n-1}{n}}[1-(1-\beta_{*})^{\frac{n-1}{n}}].\label{Eq-Reh9}
\end{align}
Using L'Hospital's rule and after a direct calculation, we see that 
\begin{align}
	\lim_{\beta_{*}\to 0}k(\beta_{*})&=1-\lim_{\beta_{*}\to 0}\left(\frac{\beta_{*}}{1-\beta_{*}}\right)^{\frac{1}{n}}=1,\label{Es-Reh8}\\
	k'(\beta_{*})&=\frac{n-1}{n}\beta_{*}^{-2+\frac{1}{n}}[1-(1-\beta_{*})^{-\frac{1}{n}}]<0.\label{Es-Reh9}
\end{align}
Hence $k(\beta_{*})$ is a decreasing function in $\beta_{*}$, which implies that $k(\beta_{*})\geq k(t)>k(1)=0$. This completes the proof of (ii).

The proof of (iii) follows directly from \eqref{Eq-Reh8} and \eqref{Es-Reh8}.

Finally, to prove \eqref{Prop-REH2}, we take $t=\frac{1}{2}$ and notice that for any measurable set $F\subset\mathbb{H}^n$, either $|F\cap B_r(x)|_g\leq\frac{1}{2}|B_r(x)|_g$ or $|F^c\cap B_r(x)|_g\leq\frac{1}{2}|B_r(x)|_g$. Then applying \eqref{Prop-REH1} to $F$ or $F^c$ with $t=\frac{1}{2}$ gives  \eqref{Prop-REH2}.
\end{proof}

Next we introduce a quantitative isoperimetric inequality proved by Bögelein, Duzaar and Scheven \cite{Bogelein15}. Given a measurable set $F\subset\mathbb{H}^n$ and $B_r$ a geodesic ball with the same volume of $F$, i.e., $|F|_g=|B_r|_g$. We define the isoperimetric gap $\mathbf{D}(F)$ to be $\mathbf{D}(F):=P(F)-P(B_r)$, which is nonnegative by the isoperimetric inequality. Also, we define the
hyperbolic Fraenkel asymmetry index $\bm{\upalpha}(F)$ as follows:
\begin{equation}\label{Defn-asym}
	\bm{\upalpha}(F)=\min_{x\in\mathbb{H}^n}|F\Delta B_r(x)|_g.
\end{equation}
\begin{thm}\label{Thm-quan}
	For any $r_0>0$, there exists a constant $c'(n,r_0)$ such that for any set $F\subset\mathbb{H}^n$ of finite perimeter with volume $|F|_g=|B_r|_g$ for some $r\in(0,r_0]$. Then the following inequality holds:
	\begin{equation}\label{In-quan}
	\mathbf{D}(F):=P(F)-P(B_r)\geq \frac{c'(n,r_0)}{\sinh^{n+1} r}\bm{\upalpha}^2(F).
	\end{equation}
\end{thm}
\begin{proof}
	The proof of Theorem \ref{Thm-quan} follows directly from combining \cite[Theorem 1.1]{Bogelein15} with \cite[Equation (1.6)]{Bogelein15}.
\end{proof}
\begin{rem}
	Indeed, \cite[Theorem 1.1]{Bogelein15} gives the inequality between $\mathbf{D}(F)$ and the $L^2$-oscillation index $\bm{\beta}(F)$ of the set $F$. Since inequality \eqref{In-quan} is enough for our use, we omit the explicit definition of $\bm{\beta}(F)$ here and readers can refer to \cite[\S 2.6]{Bogelein15} for further details.
\end{rem}
In the end of this subsection, we give an Euclidean-like isoperimetric inequality for sets in $\mathbb{H}^n$, which is proved in \cite[Lemma 2.1]{Bogelein15}.
\begin{lem}\label{Lem-Eulik}
	Let $r_0>0$. For any radius $r\in(0,r_0]$, we have the estimates 
	\begin{equation}\label{In-Eulik}
		nb_n^{\frac{1}{n}}|B_r|_g^{\frac{n-1}{n}}\leq P(B_r)\leq n b_n^{\frac{1}{n}}\left[\cosh{r_0}|B_r|_g\right]^{\frac{n-1}{n}}.
	\end{equation}
\end{lem}
\begin{rem}\label{Rem-IsoE}
	By Lemma \ref{Lem-Eulik}, the isoperimetric function $\xi$ defined in \eqref{Defn-xi} satsifies $\xi(z)\geq n b_n^{\frac{1}{n}}z^{\frac{n-1}{n}}$. Combined with Remark \ref{Rem-Iso}, there holds
	\begin{equation}\label{Eq-HEin}
		P(F)\geq \xi(|F|_g)=n b_n^{\frac{1}{n}}|F|_g^{\frac{n-1}{n}}
	\end{equation}
	for any measurable set $F\subset\mathbb{H}^n$. Also, under the same assumption of Proposition \ref{Prop-REH}, we can deduce from \eqref{Prop-REH1} and \eqref{Prop-REH2} respectively that 
	\begin{align}
		P(F;B_r(x))&\geq c_1(n,t,r_0)|F\cap B_r(x)|_g^{\frac{n-1}{n}},\label{Rem-Eulik1}\\
		P(F;B_r(x))&\geq c_2(n,r_0)\min\{|F\cap B_r(x)|_g, |B_r(x)\setminus F|_g\}^{\frac{n-1}{n}}\label{Rem-Eulik2},
	\end{align}
	where $c_1(n,t,r_0)=n b_n^{\frac{1}{n}}c(n,t,r_0)$ and $c_2(n,r_0)=n b_n^{\frac{1}{n}}c(n,r_0)$.
\end{rem}
\subsection{Quasiminimizer of the perimeter} In this subsection, we introduce two notions of quasiminimizer of the perimeter.
\begin{defn}
	A set $E\subset\mathbb{H}^n$ is said to be an $(\omega,r_0)$-minimizer for the perimeter, with $\omega>0$ and $r_0>0$, if for every geodesic ball $B_r(x)$ with $r\leq r_0$ and for every set of finite perimeter $F\subset\mathbb{H}^n$ such that $E\Delta F\subset\subset B_r(x)$, we have 
	\begin{equation}\label{Defn-Quasi}
		P(E)\leq P(F)+\omega |E\Delta F|_g.
	\end{equation}
\end{defn}
\begin{defn}\label{Defn-Quasi11}
	A set $E\subset\mathbb{H}^n$ of finite perimeter is said to be a perimeter $(K,r_1)$-quasiminimizer, with $r_1>0$ and $K\geq 1$, if for every geodesic ball $B_r(x)$ with $r\leq r_1$ and for every set of finite perimeter $F\subset\mathbb{H}^n$ such that $E\Delta F\subset\subset B_r(x)$, we have 
	\begin{equation}\label{Defn-Quasi1}
		P(E;B_r(x))\leq KP(F;B_r(x)).
	\end{equation}
\end{defn}
The following result was proved in \cite[Theorem 5.2]{Kin13} in the context of metric spaces, and the hyperbolic version was stated in \cite[Theorem 5.1]{Bogelein15}, hence we include it here for the convenience of readers.
\begin{thm}\label{Thm-quasim}
	Suppose that $E\subset\mathbb{H}^n$ is a perimeter $(K,r_1)$-quasiminimizer. Then, up to modifying $E$ in a set of measure zero, the topological boundary of $E$ coincides with the reduced boundary, i.e., $\partial E=\partial^* E$. Moreover, both $E$ and $\mathbb{H}^n\setminus E$ are locally porous in the sense that there exists a constant $C_0>1$, depending only on $n$ and $K$, such that for every $x\in\partial E$ and $0<r<r_1$, there are points $y,z\in B_r(x)$ for which
	\begin{equation*}
		B_{r/C_0}(y)\subset E\,\,\text{and}\,\, B_{r/C_0}(z)\subset \mathbb{H}^n\setminus E
	\end{equation*}
	holds true.
\end{thm}
For $F\subset\mathbb{H}^n$ measurable, we define the function $v_{F,\alpha}(\cdot)$ to be 
\begin{equation}\label{Defn-Vf}
	v_{F,\alpha}(x)=\int_F{\frac{1}{d_g(x,y)^{\alpha}}}\,dV_g(y),
\end{equation}
and hence $NL_{\alpha}(F)=\int_{F}{v_{F,\alpha}(x)}\,dV_g(x)$. We firstly prove two useful lemmas concerning $v_{F,\alpha}(\cdot)$ and $NL_{\alpha}(\cdot)$.
\begin{lem}\label{Lem-Bound}
	Given $0<\bar{m}<\infty$. If  $F\subset\mathbb{H}^n$ is a measurable set with $|F|_g\leq\bar{m}$. Then for $\forall 0<\alpha<n$, there exists a constant $c_3$ depending on $n, \bar{m}$ and $\alpha$, such that $||v_{F,\alpha}(x)||_{L^{\infty}(\mathbb{H}^n)}\leq c_3$.
\end{lem}
\begin{proof}
	By the Bathtub principle (cf. \cite[Theorem 1.14]{Leib-Loss01}), assume that $|F|_g=|B_r|_g$ for some $r>0$, then
	\begin{align}
		v_{F,\alpha}(x)&\leq\int_{B_r(x)}{\frac{1}{d_g(x,y)^{\alpha}}}\,dV_g(y)\notag\\
		&= n b_n\int_0^r\frac{\sinh^{n-1}s}{s^{\alpha}}\,ds\notag\\
		&\leq \frac{nb_n}{n-\alpha}\cosh^{n-1}(b_n^{-\frac{1}{n}}\bar{m}^{\frac{1}{n}})r^{n-\alpha}.\label{In-b1}
	\end{align}
The second inequality is due to the fact that since $b_n r^n\leq|B_r|_g\leq\bar{m}$, we have $\sinh s\leq \cosh(b_n^{-\frac{1}{n}}\bar{m}^{\frac{1}{n}})\cdot s$ for $0\leq s\leq r$. Also, \eqref{In-b1} gives
\begin{equation}\label{In-Bound}
	v_{F,\alpha}(x)\leq \frac{n b_n^{\frac{\alpha}{n}}}{n-\alpha}\cosh^{n-1}(b_n^{-\frac{1}{n}}\bar{m}^{\frac{1}{n}}){\bar{m}}^{\frac{n-\alpha}{n}}:=c_3.
\end{equation}
\end{proof}
\begin{lem}\label{Lem-NL}
	Given $0<\bar{m}<\infty$, if $E, F\subset\mathbb{H}^n$ are two measurable sets with $|E|_g, |F|_g\leq \bar{m}$, then for $\forall 0<\alpha<n$, there holds
	\begin{equation}\label{Eq-NL}
		|NL_{\alpha}(E)-NL_{\alpha}(F)|\leq 2c_3|E\Delta F|_g,
	\end{equation}
	where $c_3:=c_3(n,\bar{m},\alpha)$ is the constant given in Lemma \ref{Lem-Bound}.
\end{lem}
\begin{proof}
	Note that 
\begin{align}
	&NL_{\alpha}(E)-NL_{\alpha}(F)\notag\\
	=&\int_{\mathbb{H}^n}\int_{\mathbb{H}^n}{\left(\frac{\chi_E(x)(\chi_E(y)-\chi_F(y))}{d_g(x,y)^{\alpha}}+\frac{\chi_F(y)(\chi_E(x)-\chi_F(x))}{d_g(x,y)^{\alpha}}\right)}dV_g(x)dV_g(y)\notag\\
	=&\int_{E\setminus F}(v_{E,\alpha}(x)+v_{F,\alpha}(x))\,dV_g(x)-\int_{F\setminus E}(v_{E,\alpha}(x)+v_{F,\alpha}(x))\,dV_g(x)\notag\\
	\leq&\int_{E\Delta F}(v_{E,\alpha}(x)+v_{F,\alpha}(x))\,dV_g(x)\notag\\
	\leq&2c_3|E\Delta F|_g,\label{In-NL}
\end{align}
which implies this lemma directly.
\end{proof}
Then we turn to prove that a minimizer of the minimization problem \eqref{Prob-mini} is indeed an $(\omega,r_0)$-minimizer for a suitable choice of $(\omega, r_0)$.
\begin{prop}\label{Prop-Quasi}
	Assume that $E$ is a minimizer to the minimization problem \eqref{Prob-mini} with $|E|_g=m$. Then for any given $r_0>0$, there exists a suitable $\omega>0$ such that $E$ is an $(\omega, r_0)$-minimizer for the perimeter.
\end{prop}
\begin{proof}
	We firstly prove that there exists $\Lambda>0$ such that $E$ is a solution to the penalized problem
	\begin{equation}\label{Eq-pena}
		\min\{\mathcal{E}(F)+\Lambda\left\vert|F|_g-m\right\vert:F\in\mathbb{H}^n\}.
	\end{equation}
Then it is sufficient to show that there exists $\Lambda>0$ such that if $F\subset\mathbb{H}^n$ satisfies
\begin{equation*}
	\mathcal{E}(F)+\Lambda\left||F|_g-m\right|\leq \mathcal{E}(E),
\end{equation*}
then $|F|_g=m$. Assume by contradiction that there exists $\Lambda_h\to\infty$ and  $F_h\subset\mathbb{H}^n$ such that 
\begin{equation}\label{Eq-Quasi1}
	\mathcal{E}(F_h)+\Lambda_h\left||F_h|_g-m\right|\leq \mathcal{E}(E)\quad \text{and}\quad |F_h|_g\neq m.
\end{equation}
Using the upper-half space model $U^n$ and the transformation $\Phi_{\lambda}$  as described in \S\ref{Subsec-model}, we construct a family of new sets $\{E_h\}$ as follows:
\begin{align*}
	\Phi_{\lambda_h}: U^n&\to U^n,\notag\\
	(x_1,\dots,x_{n-1},x_n)&\mapsto(x_1,\dots,x_{n-1},\lambda_h x_n).
\end{align*}
The constants $\{\lambda_h\}$ are chosen to satsify $|E_h|_g=m$, hence by Proposition \ref{Prop-VP}, we have 
\begin{equation*}
	\lambda_h=\left(\frac{|F_h|_g}{m}\right)^{\frac{1}{n-1}}\neq 1.
\end{equation*}
Note that since $\Lambda_h\to\infty$, we have $\lambda_h\to 1$ by \eqref{Eq-Quasi1}. Next we consider the following two cases. 

	\noindent$\bullet\,\,\bf{Case\,1}.$ There exists a subsequence $\{\lambda_{h_k}\}\subset \{\lambda_h\}$ with countably many elements such that $\lambda_{h_k}>1$ and $\lambda_{h_k}\to 1$ as $k\to\infty$. In this case, note that 
	\begin{equation}\label{Ex-Quasi2}
		\mathcal{E}(E_{h_k})=P(E_{h_k})+\gamma NL_{\alpha}(E_{h_k}).
	\end{equation}
	By \eqref{Es-dis} and \eqref{Re-perimeter}, we have the following estimates:
	\begin{align}
		P(E_{h_k})&\leq P(F_{h_k}),\label{Eq-Quasi3}\\
		NL_{\alpha}(E_{{h_k}})&=\int_{E_{h_k}}\int_{E_{h_k}}{\frac{1}{(x_n')^n}\frac{1}{(y_n')^n}\frac{1}{d_g(x',y')^{\alpha}}\,dx'\,dy'}\notag\\
		&\leq \lambda_{h_k}^{-2(n-\alpha-1)}NL_{\alpha}(F_{h_k}).\label{Eq-Quasi4}
	\end{align}
	Combining \eqref{Ex-Quasi2} with \eqref{Eq-Quasi3} and \eqref{Eq-Quasi4} gives
	\begin{align}
		\mathcal{E}(E_{h_k})&\leq P(F_{h_k})+\gamma \lambda_{h_k}^{-2(n-\alpha-1)}NL_{\alpha}(F_{h_k})\notag\\
		&=\mathcal{E}(F_{h_k})+\gamma(\lambda_{h_k}^{-2(n-\alpha-1)}-1)NL_{\alpha}(F_{h_k})\notag\\
		&\leq \mathcal{E}(E)+\gamma(\lambda_{h_k}^{-2(n-\alpha-1)}-1)NL_{\alpha}(F_{h_k})-\Lambda_{h_k}\left||F_{h_k}|_g-m\right|\notag\\
		&=\mathcal{E}(E)+\gamma(\lambda_{h_k}^{-2(n-\alpha-1)}-1)NL_{\alpha}(F_{h_k})-\Lambda_{h_k}(1-\lambda_{h_k}^{1-n})|F_{h_k}|_g\notag\\
		&=\mathcal{E}(E)+(1-\lambda_{h_k}^{1-n})|F_{h_k}|_g\left[\gamma\frac{\lambda_{h_k}^{-2(n-\alpha-1)}-1}{1-\lambda_{h_k}^{1-n}}\frac{NL_{\alpha}(F_{h_k})}{|F_{h_k}|_g}-\Lambda_{h_k}\right]\notag\\
		&<\mathcal{E}(E), \quad\text{if $k$ is sufficiently large},
	\end{align}
	since 
	\begin{align*}
		\lim_{k\to\infty}{\frac{\lambda_{h_k}^{-2(n-\alpha-1)}-1}{1-\lambda_{h_k}^{1-n}}}&=-\frac{2(n-\alpha-1)}{n-1},\\
		\limsup_{k\to\infty}NL_{\alpha}(F_{h_k})&=\limsup_{k\to\infty}\int_{F_{h_k}}{v_{F_{h_k},\alpha}}\,dV_g\leq c_3 m,\\
		\lim_{k\to\infty} |F_{h_k}|_g&=m, \quad \lim_{k\to\infty} \Lambda_{h_k}=\infty,
	\end{align*}
	where the inequality is due to Lemma \ref{Lem-Bound}. Then we get a contradiction with the assumption that $E$ is a minimizer to the minimization problem \eqref{Prob-mini} and  hence that there exists $\Lambda>0$ such that $E$ is a solution to the penalized problem \eqref{Eq-pena}.

	 \noindent$\bullet\,\,\bf{Case\,2.}$ There exists a subsequence $\{\lambda_{h_k}\}\subset \{\lambda_h\}$ with countably many elements such that $\lambda_{h_k}<1$ and $\lambda_{h_k}\to 1$ as $k\to\infty$. In this case, we can argue as $\bf{Case\,1}$ to deduce that $E$ is a solution to the penalized problem \eqref{Eq-pena}.

Combining the two cases above, given any $r_0>0$ and for every geodesic ball $B_r(x)$ with $r\leq r_0$ and every set of finite perimeter $F\subset\mathbb{H}^n$ such that $E\Delta F\subset\subset B_r(x)$, we have 
\begin{align}
	P(E)&\leq P(F)+\gamma(NL_{\alpha}(F)-NL_{\alpha}(E))+\Lambda\left||F|_g-m\right|\notag\\
	&\leq P(F)+(2\gamma c_3+\Lambda)|E\Delta F|_g,\label{In-Quasi6}
\end{align}
where we have used the fact that $||F|_g-m|\leq |E\Delta F|_g$ and Lemma \ref{Lem-NL} with $\bar{m}=m+|B_{r_0}|_g$. This completes the proof with $\omega=2\gamma c_3+\Lambda$ for any given positive $r_0$.
\end{proof}
\subsection{Fuglede’s type estimates in $\mathbb{H}^n$} In this subsection, we give Fuglede’s type estimates of $P(E)$ and $NL_{\alpha}(E)$ for a nearly spherical set $E\subset\mathbb{H}^n$, which are crucial in the proof of Theorem \ref{Thm-Existence}.
\begin{thm}\label{Thm-Fuglede}
	For any $n\geq 2$, $0<\alpha<n$ and $R_0>0$, there exists a constant $\varepsilon_0\in(0,\frac{1}{2}]$ depending only on $n$ and $R_0$, such that if $E\subset\mathbb{H}^n$ is a nearly spherical set with $|E|_g=|B_r|_g$ for some $r\in(0,R_0]$, whose boundary is given by 
	\begin{equation}\label{Eq-Fuglede}
		\partial E=\{(x,r(1+u(x))):x\in\mathbb{S}^{n-1}\}
	\end{equation}
	under the geodesic polar coordinate, where $u:\mathbb{S}^{n-1}\to\mathbb{R}$ is a $C^1$ function, with $||u||_{C^1(\mathbb{S}^{n-1})}\leq\varepsilon_0$. Then we have the estimate 
	\begin{equation}\label{Eq-F1}
		NL_{\alpha}(B_r)-NL_{\alpha}(E)\leq L_0{\sinh}^{2n-\alpha} r\left([u]^2_{\frac{\alpha+1-n}{2}}+||u||^2_{L^2(
			\mathbb{S}^{n-1})}\right)
	\end{equation}
	for some constant $L_0>0$ depending only on $n,\alpha$ and $R_0$. Here $[u]^2_{\frac{\alpha+1-n}{2}}$ is defined by
	\begin{equation*}
		[u]^2_{\frac{\alpha+1-n}{2}}=\int_{\mathbb{S}^{n-1}}\int_{\mathbb{S}^{n-1}}\frac{|u(x)-u(y)|^2}{|x-y|^{\alpha}}\,d\sigma_x\,d\sigma_y.
	\end{equation*}
\end{thm}  
\begin{proof}
	We slightly change the notation, we assume that $|E_t|_g=|B_r|_g$, where 
	\begin{equation*}
		\partial E_t=\{(x,r(1+tu(x))):x\in\mathbb{S}^{n-1}\},
\end{equation*}
which satisfies $||u||_{C^1(\mathbb{S}^{n-1})}\leq\frac{1}{2}$ and $t\in(0,2\varepsilon_0]$. Then for any two points $\tilde{x}=(x,\rho)$ and $\tilde{y}=(y,s)$ under the geodesic polar coordinate, the distance $d_g(\tilde{x},\tilde{y})$ satisfies
\begin{equation}\label{Eq-F3}
	\cosh d_g(\tilde{x},\tilde{y})=\cosh\rho\cosh s-\sinh\rho\sinh s\cos\beta,
\end{equation}
where we have used the hyperbolic cosine law and $\beta$ is the angle between $x$ and $y$,
which satisfies
\begin{equation}\label{Eq-F4}
	\cos\beta=\frac{2-|x-y|^2}{2}.
\end{equation}
Combining \eqref{Eq-F3} with \eqref{Eq-F4}, we have
\begin{align*}
	\cosh d_g(\tilde{x},\tilde{y})&=\cosh\rho\cosh s-\sinh\rho\sinh s+\sinh\rho\sinh s(1-\cos\beta)\\
	&=\cosh(\rho-s)+\frac{\sinh\rho\sinh s}{2}|x-y|^2,
\end{align*} 
which gives
\begin{equation*}
	d_g(\tilde{x},\tilde{y})=\mathrm{arccosh}\left[\cosh(\rho-s)+\frac{\sinh\rho\sinh s}{2}|x-y|^2\right].
\end{equation*}
If we define the function $f_{\theta}(\rho,s)$ as
\begin{equation}\label{Eq-F5}
	f_{\theta}(\rho,s)=\frac{\sinh^{n-1}\rho\sinh^{n-1}s}{\mathrm{arccosh}^{\alpha}\left[\cosh(\rho-s)+\frac{\sinh\rho\sinh s}{2}{\theta}^2\right]}.
\end{equation}
Then $NL_{\alpha}(E_t)$ can be expressed by 
\begin{equation}\label{Eq-F6}
	NL_{\alpha}(E_t)=\int_{\mathbb{S}^{n-1}}\,d\sigma_x\int_{\mathbb{S}^{n-1}}\,d\sigma_y\int_0^{r(1+tu(x))}\int_0^{r(1+tu(y))}{f_{|x-y|}(\rho,s)}\,d\rho\,ds.
\end{equation}
By exploiting the identity
\begin{equation*}
	2\int_0^a\int_0^b=\int_0^a\int_0^a+\int_0^b\int_0^b-\int_a^b\int_a^b,
\end{equation*}
we deduce that 
\begin{align}
	NL_{\alpha}(E_t)=&\int_{\mathbb{S}^{n-1}}\,d\sigma_x\int_{\mathbb{S}^{n-1}}\,d\sigma_y\int_0^{r(1+tu(x))}\int_0^{r(1+tu(x))}{f_{|x-y|}(\rho,s)}\,d\rho\,ds\notag\\
	&-\frac{1}{2}\int_{\mathbb{S}^{n-1}}\,d\sigma_x\int_{\mathbb{S}^{n-1}}\,d\sigma_y\int_{r(1+tu(y))}^{r(1+tu(x))}\int_{r(1+tu(y))}^{r(1+tu(x))}{f_{|x-y|}(\rho,s)}\,d\rho\,ds\notag\\
	=:&I-II.\label{Eq-F7}
\end{align}
By taking $u=0$ in \eqref{Eq-F7}, we have
\begin{equation}\label{Eq-F8}
	NL_{\alpha}(B_r)=\int_{\mathbb{S}^{n-1}}\,d\sigma_x\int_{\mathbb{S}^{n-1}}\,d\sigma_y\int_0^{r}\int_0^{r}{f_{|x-y|}(\rho,s)}\,d\rho\,ds.
\end{equation}
Combining \eqref{Eq-F7} with \eqref{Eq-F8} gives
\begin{equation}\label{Eq-F9}
		NL_{\alpha}(B_r)-NL_{\alpha}(E_t)=NL_{\alpha}(B_r)-I+\frac{t^2 r^2}{2}h(t),
\end{equation}
where $h(t)$ is defined as 
\begin{equation}\label{Eq-F10}
	h(t)=\int_{\mathbb{S}^{n-1}}\,d\sigma_x\int_{\mathbb{S}^{n-1}}\,d\sigma_y\int_{u(y)}^{u(x)}\int_{u(y)}^{u(x)}{f_{|x-y|}(r(1+t\rho),r(1+ts))}\,d\rho\,ds.
\end{equation}
By Lemma \ref{Lem-appen_0} and \ref{Lem-appen}, there exists two constants $L'$ and $L''$ depending on $n,\alpha$ and $R_0$, such that  
\begin{align}
	&NL_{\alpha}(B_r)-I\leq t^2L' \sinh^{2n-\alpha}r||u||^2_{L^2(\mathbb{S}^{n-1})},\label{Eq-F10.2}\\
	&\left|\frac{\partial f_{|x-y|}(r(1+\tau\rho),r(1+\tau s))}{\partial\tau}\right|\leq L''f_{|x-y|}(r,r),\,\,\forall\tau\in(0,t)\label{Eq-F10.5}.
\end{align}
From \eqref{Eq-F10.5} we have
\begin{equation}\label{Eq-F11}
	|h'(\tau)|\leq L''h(0),\,\,\forall\tau\in(0,t).
\end{equation}
Next we estimate $h(0)$. Let $w=\mathrm{arccosh}\left[1+\frac{\sinh^2 r}{2}|x-y|^2\right]$, i.e. $\cosh w=1+\frac{\sinh^2 r}{2}|x-y|^2$. Since $|x-y|\leq 2$, $r\leq R_0$, there exists $L'''(R_0)$, such that $\cosh w\leq 1+\frac{L'''}{2}w^2$, which implies
\begin{equation}\label{Eq-F12}
	w\geq\frac{\sinh r}{\sqrt{L'''}}|x-y|.
\end{equation}
Combining \eqref{Eq-F10}-\eqref{Eq-F12}, we get
\begin{align}
	h(t)&\leq \left(1+tL''\right)h(0)\notag\\
	&=(1+tL'')\int_{\mathbb{S}^{n-1}}\,d\sigma_x\int_{\mathbb{S}^{n-1}}\,d\sigma_y\int_{u(y)}^{u(x)}\int_{u(y)}^{u(x)}{f_{|x-y|}(r,r)}\,d\rho\,ds\notag\\
	&\leq (1+tL'')(L''')^{\frac{\alpha}{2}}\sinh^{2n-2-\alpha} r [u]^2_{\frac{\alpha+1-n}{2}}.\label{Eq-F13}
\end{align}
Then combining \eqref{Eq-F9}, \eqref{Eq-F10.2} and \eqref{Eq-F13} gives
\begin{equation}
	NL_{\alpha}(B_r)-NL_{\alpha}(E_t)\leq t^2\sinh^{2n-\alpha} r\left[L'||u||^2_{L^2(\mathbb{S}^{n-1})}+\frac{1+tL''}{2}(L''')^{\frac{\alpha}{2}}[u]^2_{\frac{\alpha+1-n}{2}}\right],
\end{equation}
which obviously implies the announced result.
\end{proof}
\begin{rem}
	Similar estimate in $\mathbb{R}^n$ has been proved in \cite[Lemma 5.3]{Figalli15}.
\end{rem}
\begin{rem}\label{Rem-R1}
	There holds $[u]^2_{\frac{\alpha+1-n}{2}}\leq C(n,\alpha) ||\nabla u||^2_{L^2(\mathbb{S}^{n-1})}$ according to the arguments in \cite[Pages 486-490]{Figalli15}, hence \eqref{Eq-F1} also implies that there exists a constant $L_1$, depending only on $n$, $\alpha$ and $R_0$, such that
	\begin{equation}\label{Eq-F2}
			NL_{\alpha}(B_r)-NL_{\alpha}(E)\leq L_1 \sinh^{2n-\alpha} r||u||^2_{W^{1,2}(\mathbb{S}^{n-1})}.
	\end{equation}
\end{rem}
\begin{rem}\label{Rem-Fuglede}
	Under an additional assumption that the barycenter of $E$ is in the origin, Bögelein, Duzaar and Scheven (cf. \cite[Theorem 4.1]{Bogelein15}) have proved an Fuglede's type estimate for the perimeter, i.e., there exists a constant $\varepsilon_1\in(0,\frac{1}{2}]$ depending only on $n$ and $R_0$, such that if $||u||_{C^1(\mathbb{S}^{n-1})}\leq\varepsilon_1$, there holds
	\begin{equation}
		P(E)-P(B_r)\geq L_2P(B_r)||u||^2_{W^{1,2}(\mathbb{S}^{n-1})}=nb_nL_2\sinh^{n-1}r||u||^2_{W^{1,2}(\mathbb{S}^{n-1})}
	\end{equation}
for some constant $L_2$ depending only on $n$ and $R_0$.
\end{rem}
\section{Basic properties of minimizers}\label{Sec-Basic es}
Assume that $E\subset\mathbb{H}^n$ is a minimizer solving the minimization problem \eqref{Prob-mini} with $|E|_g=m$. In this section, we prove that $\partial^{*} E$ is $C^{1,\frac{1}{2}}$, $E$ is essentially bounded, indecomposable and admits a uniform lower density bound.
\begin{prop}
	Let $E\subset\mathbb{H}^n$ be a minimizer solving the minimization problem \eqref{Prob-mini} with $|E|_g=m$. Then the reduced boundary $\partial^{*}E$ is a $C^{1,\frac{1}{2}}$ manifold and the Hausdorff dimension of the singular set satisfies $\text{dim}_H(\partial E\setminus\partial^{*}E)\leq n-8$.
\end{prop}
\begin{proof}
	By Proposition \ref{Prop-Quasi}, we see that $E$ is an $(\omega, r_0)$-minimizer for the perimeter for two positive finite constants $\omega$ and $r_0$. Then we can apply the result of \cite[Thorem 1]{Tamanini82} to get the desired regularity and the Hausdorff dimension of the singular set.
\end{proof}

It is well-known that the $(\omega,r_0)$-minimizer satisfies the upper and lower density estimates. We state it in the following theorem, whose proof is a direct adaptation of \cite[Theorem 6,(i),(iii)]{Caraballo11} and \cite[Theorem 21.11]{Maggi-12} from $\mathbb{R}^n$ to $\mathbb{H}^n$ case due to the local nature.
\begin{thm}\label{Lem-twoside}
	Assume that $F\subset\mathbb{H}^n$ is an $(\omega,r_0)$-minimizer for the perimeter in $\mathbb{H}^n$, then there exists a constant $r'\in(0,r_0)$,  such that for any $0<r\leq r'$ and $x\in\partial F$, there holds
	\begin{equation}\label{Eq-twoside}
		0<c_4\leq\frac{|F\cap B_r(x)|_g}{|B_r(x)|_g}\leq 1-c_4<1
	\end{equation}
	for some constant $c_4\in(0,1)$ independent of $x$. Moreover, both $\mathring{F}^M$ and $(\mathring{F}^c)^M$ are open.
\end{thm}


We also present the relationship between $\partial^M F$, $\partial\mathring{F}^M$, $\partial^{*}F$ and $\partial F$ for a set of finite perimeter $F\subset\mathbb{H}^n$.

\begin{lem}\label{Lem-Re-bou}
	Let $F\subset\mathbb{H}^n$ be a set of finite perimeter, then we have
	\begin{equation}\label{Eq-Re-bou}
	\partial^{*} F\subset \partial^M F\subset \partial\mathring{F}^M=\overline{\partial^{*} F}\subset\partial F.
	\end{equation}
	Moreover, if $F$ satisfies the density estimate \eqref{Eq-twoside}, then 
	\begin{equation}\label{Eq-Re-bou1}
		\partial^{*} F\subset \partial^M F=\partial\mathring{F}^M=\overline{\partial^{*} F}\subset\partial F.
	\end{equation}
\end{lem}
\begin{proof}
	The proof of $\partial^{*} F\subset \partial^M F$ can be found in \cite[Theorem 3.61]{Ambrosio-Fusco-Pallara00}, and $\partial\mathring{F}^M=\overline{\partial^{*} F}$ can be found in \cite[Theorem 10]{Caraballo11}. By the definition of $\partial^M F$, we see that for $\forall x\in\partial^M F$ and $\forall r>0$, 
	\begin{align*}
		|B_r(x)\cap F|_g&=|B_r(x)\cap\mathring{F}^M|_g>0,\\
		|B_r(x)\cap F^c|_g&=|B_r(x)\cap(\mathring{F}^c)^M|_g>0,
	\end{align*}
	which implies that $x\in\partial\mathring{F}^M$ and hence $\partial^M F\subset \partial\mathring{F}^M$. $\overline{\partial^{*} F}\subset\partial F$ follows from the fact that $\partial^{*} F\subset\partial F$ and $\partial F$ is closed. This completes the proof of \eqref{Eq-Re-bou}.
	
	If $F$ satisfies the density estimate \eqref{Eq-twoside}, then by \eqref{Eq-Essenb} and Theorem \ref{Lem-twoside}, the set $\partial^M F$ is closed. Hence we have $\overline{\partial^{*} F}\subset \partial^M F$, which gives \eqref{Eq-Re-bou1}.
\end{proof}

Then we can prove that a minimizer of the functional $\mathcal{E}(\cdot)$ is essentially bounded and indecomposable.
\begin{thm}\label{Thm-EI}
	Assume that $E\subset\mathbb{H}^n$ solves the minimization problem \eqref{Prob-mini} with $|E|_g=m$. Then $E$ is essentially bounded and indecomposable.
\end{thm}
\begin{proof}
	Firstly, we prove that $E$ is essentially bounded. By Proposition \ref{Prop-Quasi} and Theorem \ref{Lem-twoside}, we find that there exists $r>0$ such that for $\forall x\in\partial E$, we have 
	\begin{equation}\label{Eq-ess1}
	|E\cap B_r(x)|_g\geq c_4 |B_r(x)|_g:=c_4 |B_r|_g
	\end{equation}
	for some constant $c_4$ independent of $x$. Then we argue by contradiction. Assume that there exists a sequence of points $\{x_n\}\subset\bar{E}^M$, such that $d_g(x_n,O)\to\infty$ as $n\to\infty$, where $O\in\mathbb{H}^n$ is a fixed point. Without loss of generality, we assume that $d_g(x_n,x_m)>4r$ and hence $B_r(x_n)\cap B_r(x_m)=\emptyset$ for $n\neq m$. By the definition of $\partial^M E$, we see that for any $n$, 
	\begin{equation}\label{Eq-ess2}
	  	|B_r(x_n)\cap \bar{E}^M|_g>0,\quad \text{and}\quad	|B_r(x_n)\setminus \bar{E}^M|_g>0.
	\end{equation}
	This implies that there exists $y_n\in\partial^M E\cap B_r(x_n)$ for any $n$, since otherwise by \eqref{Eq-Essenb}, we have
	\begin{equation*}
		B_r(x_n)=\left(B_r(x_n)\cap\mathring{E}^M\right)\cup \left(B_r(x_n)\cap(\mathring{E}^c)^M\right),
	\end{equation*}
	which says that the geodesic ball $B_r(x_n)$ is the union of two non-empty disjoint open sets by Proposition \ref{Prop-Quasi} and Theorem \ref{Lem-twoside}, contradicting the connectedness of $B_r(x_n)$. By Lemma \ref{Lem-Re-bou}, we have $\partial^M E=\partial \mathring{E}^M=\overline{\partial^{*}E}\subset\partial E$, hence the points $\{y_n\}$ satisfy the lower density bound \eqref{Eq-ess1}, and then we have
	\begin{equation*}
		m=|E|_g\geq\sum_{n=1}^{\infty}|E\cap B_r(y_n)|_g\geq c_4\sum_{n=1}^{\infty}|B_r|_g=\infty,
	\end{equation*}
	which is a contradiction and the essentially boundedness of $E$ is proved.
	
	Next we prove that $E$ is indecomposable. We assume by contradiction that if $E$ is decomposable, i.e. there exists a partition $(E_1,E_2)$ of  $E$ with both $|E_1|_g>0$ and $|E_2|_g>0$, such that $P(E)=P(E_1)+P(E_2)$. Take $\{\tau_h\}$ be a family of isometries in $\mathbb{H}^n$, which satisfies 
  \begin{equation}
  	d_g(E_1,\tau_h E_2)\to\infty, \quad \text{as} \quad h\to\infty.
  	\end{equation}
  	Since $E$ is essentially bounded, hence $E_1$ and $E_2$ are also essentially bounded. Take $\bar{E}_h=E_1\cup \tau_h E_2$, then for $h$ sufficiently large, we have
  	\begin{align*}
  		P(\bar{E}_h)&=P(E_1)+P(\tau_h E_2)=P(E_1)+P(E_2)=P(E),\\
  		|\bar{E}_h|_g&=|E_1|_g+|\tau_h E_2|_g=|E_1|_g+|E_2|_g=|E|_g,
  	\end{align*}
  	and hence
  	\begin{align*}
  		\liminf_{h\to\infty}{\mathcal{E}}(\bar{E}_h)&=\liminf_{h\to\infty}(P(\bar{E}_h)+\gamma NL_{\alpha}(\bar{E}_h))\\
  		&=P(E)+\gamma NL_{\alpha}(E_1)+\gamma NL_{\alpha}(E_2)\\
  		&<P(E)+\gamma NL_{\alpha}(E_1)+\gamma NL_{\alpha}(E_2)+2\gamma\int_{E_1}\int_{E_2}{\frac{1}{d_{g}(x,y)^{\alpha}}\,dV_g(x)\,dV_g(y)}\\
  		&=\mathcal{E}(E),
  	\end{align*}
  	which implies that there exists some sufficiently large $h_0$, such that
  \begin{equation}
  	|\bar{E}_{h_0}|_g=|E|_g,\quad {\mathcal{E}}(\bar{E}_{h_0})<\mathcal{E}(E).
  \end{equation}
  This contradicts the fact that $E$ is a minimizer of the functonal $\mathcal{E}(\cdot)$ with volume $m$, hence $E$ is indecomposable.
\end{proof}

Let $F\subset\mathbb{H}^n$ be a set of finite perimeter with $|F|_g=m$. Denote $\bar{r}=\bar{r}(n)$ to be radius such that $|B_{\bar{r}}|_g=1$. Then for $R>2\bar{r}$, we construct the set $\widehat{F}_R$ to be the collection of $N\geq 1$ balls of equal size and located at $x_k=\mathrm{e}^{Rk}e_n\,1\leq k\leq n$. Here we use the upper-half space model $U^n$ (cf. \S\ref{Subsec-model}). The number $N$ is chosen to be the smallest integer for which the volume of each ball does not exceed 1 and such that $|\widehat{F}_R|_g=|F|_g=m$. Note that since $d_g(x_k,x_{k+1})=R$ by \eqref{Ex-metric}, and hence the $N$ balls in $\widehat{F}_R$ are disjoint from each other. 
\begin{lem}\label{Lem-ER}
	Let $F\subset\mathbb{H}^n$ be a set of finite perimeter with $|F|_g=m$. Then we have
 \begin{equation}\label{Eq-ER}
	\limsup_{R\to\infty}\mathcal{E}({\widehat{F}_R})\leq c_5\max\{m,m^{\frac{n-1}{n}}\},
\end{equation} 
where $c_5$ is a constant depending on $n, \alpha$ and $\gamma$.
\end{lem}
\begin{proof}
	Denote the geodesic ball centered at $x_k$ as $B^k$ respectively, then for $1\leq i\neq j\leq N$,
	\begin{equation}
		d_g(B^i, B^j)\to\infty, \quad \text{as $R\to\infty$}.
	\end{equation}
	We firstly estimate the perimeter term $P(\widehat{F}_R)$. 
	
	\noindent(i) If $m\leq 1$, then $\widehat{F}_R$ is precisely a geodesic ball $B_r$ with radius $r\in(0,\bar{r}]$. Then by Lemma \ref{Lem-Eulik}, we have 
	\begin{equation}\label{eq-ER1}
		P(\widehat{F}_R)=P(B_r)\leq n b_n^{\frac{1}{n}}(\cosh\bar{r})^{\frac{n-1}{n}}m^{\frac{n-1}{n}}.
	\end{equation}
   (ii) If $m>1$, then we have $P(\widehat{F}_R)\leq N P(B_{\bar{r}})$,
   where $N\leq m+1=\frac{m+1}{m}\cdot m\leq 2m$ by the definition of $N$. Hence
   \begin{equation}\label{eq-ER2}
   	P(\widehat{F}_R)\leq 2P(B_{\bar{r}}) m\leq 2n b_n^{\frac{1}{n}}(\cosh\bar{r})^{\frac{n-1}{n}} m.
   	\end{equation}
   	Combining \eqref{eq-ER1} with \eqref{eq-ER2} gives
   	\begin{equation}\label{eq-ER3}
   		P(\widehat{F}_R)\leq  2n b_n^{\frac{1}{n}}(\cosh\bar{r})^{\frac{n-1}{n}}\max\{m,m^{\frac{n-1}{n}}\}.
   	\end{equation} 
   	Next we consider the term $NL_{\alpha}(\widehat{F}_R)$. Note that
   	\begin{align*}
   		NL_{\alpha}(\widehat{F}_R)&=\int_{\widehat{F}_R}\int_{\widehat{F}_R}\frac{1}{d_g(x,y)^{\alpha}}\,dV_g(x)\,dV_g(y)\\
   		&=\sum_{i,j=1}^N\int_{B^i}\int_{B^j}\frac{1}{d_g(x,y)^{\alpha}}\,dV_g(x)\,dV_g(y).
   \end{align*}
   Hence 
   \begin{equation}\label{eq-ER4}
   \limsup_{R\to\infty}NL_{\alpha}(\widehat{F}_R)=\sum_{i=1}^N\int_{B^i}\int_{B^i}\frac{1}{d_g(x,y)^{\alpha}}\,dV_g(x)\,dV_g(y).
   \end{equation}
  Note that the right hand side of \eqref{eq-ER4} keeps invariant regardless of where the center $x_k$ is, so we omit it in the expression as $R\to\infty$. Recall Lemma \ref{Lem-Bound}, we have $	v_{B^i,\alpha}(x)\leq c_3$ for some $c_3$ depending only on $n$ and $\alpha$ since $|B^i|_g\leq 1$. It gives
  \begin{equation}\label{eqER6}
  	\limsup_{R\to\infty}NL_{\alpha}(\widehat{F}_R)\leq \sum_{i=1}^N\int_{B^i}c_3\,dV_g(x)=c_3|B^1|_g N.
  \end{equation}
 As in the estimate of the perimeter term, we also consider two cases:
 
 \noindent(i) If $m\leq 1$, then $|B^1|_g=m$ and $N=1$, then we have
 \begin{equation}\label{eqER7}
 	\limsup_{R\to\infty}NL_{\alpha}(\widehat{F}_R)\leq c_3m.
 \end{equation}
 (ii) If $m>1$, then $|B^1|_g\leq |B_{\bar{r}}|_g=1$ and $N\leq 2m$, which gives
 \begin{equation}\label{eqER8}
 	\limsup_{R\to\infty}NL_{\alpha}(\widehat{F}_R)\leq 2c_3m.
 \end{equation}
 Combining \eqref{eq-ER3} with \eqref{eqER7} and \eqref{eqER8} proves Lemma \ref{Lem-ER} with the constant $c_5=2n b_n^{\frac{1}{n}}(\cosh\bar{r})^{\frac{n-1}{n}}+2\gamma c_3$.
\end{proof}

Our next two lemmas are adaptations of \cite[Lemma 4.2, Lemma 4.3]{Knupfer-Muratov14}. The first is a general criterion on a set of finite perimeter being not a minimizer.
\begin{lem}[Nonoptimality Criterion]\label{Lem-Cri}
	Let $F\subset\mathbb{H}^n$ be a set of finite perimeter with $|F|_g=m$, suppose that there exist two disjoint sets of finite perimeter $F_1, F_2$ with positive measures such that $F=F_1\cup F_2$ and
	\begin{equation}\label{Lem-Cri1}
		\Sigma:=P(F_1)+P(F_2)-P(F)\leq\frac{1}{2}\mathcal{E}(F_2).
	\end{equation}
	Then there is an $\varepsilon>0$ depending only $n,\alpha$ and $\gamma$, such that if
	\begin{equation}\label{Lem-Cri2}
		|F_2|_g\leq\varepsilon\min\{1,|F_1|_g\},
	\end{equation}
there exists a set $G\subset\mathbb{R}^n$ such that $|G|_g=|F|_g$ and $\mathcal{E}(G)<\mathcal{E}(F)$. 
\end{lem}
\begin{proof}
 Let $m_1=|F_1|_g$ and $m_2=|F_2|_g$. Using the upper-half space model $U^n$ and let $\Phi_{\lambda}$ be the transformation defined in \eqref{Defn-Phi}, which maps $F_1$ to the set $\widetilde{F}$. We choose $\lambda$ such that $|\widetilde{F}|_g=|F|_g=m$, then by \eqref{Re-volume}, we deduce that $\lambda^{1-n}=1+\frac{m_2}{m_1}$. Let $\beta:=\frac{m_2}{m_1}\leq\varepsilon$ by \eqref{Lem-Cri2}. Then we have 
 \begin{equation}\label{Eq-Cri1}
 	\lambda=(1+\beta)^{-\frac{1}{n-1}}<1.
 \end{equation}
  We will compare $\mathcal{E}(F)$ with $\mathcal{E}(\widetilde{F})$ and $\mathcal{E}({\widehat{F}_R})$ to get suitable choice of $G$. Firstly, if $\limsup_{R\to\infty}\mathcal{E}({\widehat{F}_R})<\mathcal{E}(F)$, then there exists $R_0$ sufficiently large, such that $G:=\widehat{F}_{R_0}$ satisfies the requirements of this lemma. Otherwise, we assume that 
  \begin{equation}\label{Eq-Cri2}
  	\mathcal{E}(F)\leq \limsup_{R\to\infty}\mathcal{E}({\widehat{F}_R})\leq c_5\max\{m,m^{\frac{n-1}{n}}\}.
  \end{equation}
  Then it remains to show that there is a sufficiently small $\varepsilon>0$ depending on $n,\alpha$ and $\gamma$, such that if the condition \eqref{Lem-Cri2} holds, then $\mathcal{E}(\widetilde{F})<\mathcal{E}(F)$. By \eqref{Es-dis}, \eqref{Re-perimeter} and \eqref{Eq-Cri1}, we see that 
  \begin{align}
  	P(\widetilde{F})&\leq\lambda^{2-2n}P(F_1),\label{Eq-Cri3}\\
  	NL_{\alpha}(\widetilde{F})&=\int_{\widetilde{F}}\int_{\widetilde{F}}\frac{1}{(x_n')^n}\frac{1}{(y_n')^n}\frac{1}{d_g(x',y')^{\alpha}}\,dx'\,dy'\notag\\
  	&\leq {\lambda}^{2-2n} NL_{\alpha}(F_1).\label{Eq-Cri4}
  \end{align} 
  Combining \eqref{Eq-Cri3} with \eqref{Eq-Cri4} gives
  \begin{align}
  	\mathcal{E}(\widetilde{F})&=P(\widetilde{F})+\gamma NL_{\alpha}(\widetilde{F})\notag\\
  	&\leq {\lambda}^{2-2n} \mathcal{E}(F_1)\notag\\
  	&=\mathcal{E}(F_1)+({\lambda}^{2-2n}-1)\mathcal{E}(F_1).\label{Eq-Cri5}
  \end{align}
  If we choose $\varepsilon\leq 1$, then 
  \begin{equation*}
  	    1\leq\lambda^{-1}=(1+\beta)^{\frac{1}{n-1}}\leq (1+\varepsilon)^{\frac{1}{n-1}}\leq 2^{\frac{1}{n-1}}.
  \end{equation*}
  Hence, there exists a constant $K$ depending on $n$ such that 
 \begin{equation}\label{Eq-Cri6}
 	\lambda^{2-2n}-1\leq K(\lambda^{-1}-1)=K[(1+\beta)^{\frac{1}{n-1}}-1]\leq K\beta,
 \end{equation}
 where the first inequality is due to Langrange's mean value theorem. Combining \eqref{Eq-Cri5} with \eqref{Eq-Cri6} gives
 \begin{equation}\label{Eq-Cri7}
 	\mathcal{E}(\widetilde{F})\leq (1+K\beta)\mathcal{E}(F_1).
 \end{equation}
 By assumption \eqref{Lem-Cri1}, we have 
 \begin{align}
 	&\mathcal{E}(F)=P(F)+\gamma NL_{\alpha}(F)\notag\\
 	>&P(F)+\gamma NL_{\alpha}(F_1)+\gamma NL_{\alpha}(F_2)\notag\\
 	\geq&P(F_1)+P(F_2)-\frac{1}{2}\mathcal{E}(F_2)+\gamma NL_{\alpha}(F_1)+\gamma NL_{\alpha}(F_2)\notag\\
 	=&\mathcal{E}(F_1)+\frac{1}{2}\mathcal{E}(F_2).\label{Eq-Cri8}
 \end{align}
 From \eqref{Eq-Cri7} and \eqref{Eq-Cri8}, we get
 \begin{equation}\label{Eq-Cri9}
 		\mathcal{E}(\widetilde{F})-\mathcal{E}(F)\leq -\frac{1}{2}\mathcal{E}(F_2)+K\beta\mathcal{E}(F_1).
 \end{equation}
 By \eqref{Eq-HEin}, we have
 \begin{equation}\label{Eq-Cri10}
 	\mathcal{E}(F_2)>P(F_2)\geq n b_n^{\frac{1}{n}}m_2^{\frac{n-1}{n}}>0.
 \end{equation}
 Also, by \eqref{Eq-Cri2} and \eqref{Eq-Cri8},
 \begin{equation}\label{Eq-Cri11}
 	\mathcal{E}(F_1)<\mathcal{E}(F)\leq c_5\max\{m,m^{\frac{n-1}{n}}\}.
 \end{equation}
 Combining \eqref{Eq-Cri9}-\eqref{Eq-Cri11}, we arrive at
 \begin{align}
 	\mathcal{E}(\widetilde{F})-\mathcal{E}(F)&\leq -\frac{1}{2}n b_n^{\frac{1}{n}}m_2^{\frac{n-1}{n}}+K\beta c_5\max\{m,m^{\frac{n-1}{n}}\}\notag\\
 	&:=-d_1(n)m_2^{\frac{n-1}{n}}+d_2(n,\alpha,\gamma)\beta\max\{m,m^{\frac{n-1}{n}}\}.\label{Eq-Cri12}
 \end{align}
 Since by \eqref{Lem-Cri2}, $\beta,m_2\leq\varepsilon$, then 
 \begin{equation*}
 	\beta m=\frac{m_2}{m_1}(m_1+m_2)=(1+\frac{m_2}{m_1})m_2\leq (1+\varepsilon)m_2\leq 2m_2.
 \end{equation*}
 Then \eqref{Eq-Cri12} now becomes
 \begin{align}
 	\mathcal{E}(\widetilde{F})-\mathcal{E}(F)&\leq -d_1(n)m_2^{\frac{n-1}{n}}+d_2(n,\alpha,\gamma)\max\{2m_2,{\varepsilon}^{\frac{1}{n}}(2m_2)^{\frac{n-1}{n}}\}\notag\\
 	&=-d_1(n)m_2^{\frac{n-1}{n}}+2^{\frac{n-1}{n}}d_2(n,\alpha,\gamma)m_2^{\frac{n-1}{n}}\max\{(2m_2)^{\frac{1}{n}},\varepsilon^{\frac{1}{n}}\}\notag\\
 	&=-d_1(n)m_2^{\frac{n-1}{n}}+2d_2(n,\alpha,\gamma)\varepsilon^{\frac{1}{n}}m_2^{\frac{n-1}{n}}\notag\\
 	&=(-d_1(n)+2d_2(n,\alpha,\gamma)\varepsilon^{\frac{1}{n}})m_2^{\frac{n-1}{n}}.
 \end{align}
If we take 
\begin{equation}\label{Value-e}
	\varepsilon=\min\left\{1,\left[\frac{d_1(n)}{4d_2(n,\alpha,\gamma)}\right]^n\right\},
\end{equation}
then we have $\mathcal{E}(\widetilde{F})<\mathcal{E}(F)$.
\end{proof}

In the end of this section, we improve the lower density result of Theorem \ref{Lem-twoside} to a bound which independent of the choice of set.
\begin{lem}[Uniform lower density bound]\label{Lem-ULB}
	Let $E\subset\mathbb{H}^n$ be a set of finite perimeter, which solves the minimization problem \eqref{Prob-mini} with $|E|_g=m$. Then for every $x\in\bar{E}^M$, we have for some constant $c_6$ depending on $n,\alpha$ and $\gamma$, there holds
	\begin{equation}\label{Lem-ULB!}
		|E\cap B_1(x)|_g\geq c_6\min\{1,m\}.
	\end{equation}
\end{lem}
\begin{proof}
	For given $r>0$ and $x\in\bar{E}^M$, define the sets $F_1^r=E\setminus B_r(x)$ and $F_2^r=E\cap B_r(x)$.  Hence by $P(B_r)=n b_n\sinh^{n-1} r$ and \eqref{In-Eulik}, we have
	\begin{equation}\label{Eq-ULB1}
		|F_2^r|_g\leq |B_r(x)|_g\leq n^{-\frac{n}{n-1}}b_n^{-\frac{1}{n-1}}P(B_r(x))^{\frac{n}{n-1}}=b_n\sinh^n r.
	\end{equation}
	If we take
	\begin{equation}\label{Eq-ULB2}
		d_3=\min\left\{\mathrm{arccosh}(2^{\frac{1}{n}}),\left[\frac{\varepsilon}{2b_n(1+\varepsilon)}\right]^{\frac{1}{n}}\right\}<1,
	\end{equation}
	where $\varepsilon$ is the constant appeared in Lemma \ref{Lem-Cri} (we can choose $\varepsilon$ as in \eqref{Value-e}). Then condition \eqref{Lem-Cri2} holds for any $r\leq r_s:=d_3\min\{1,m^{\frac{1}{n}}\}$. Indeed, since $r\leq d_3\leq {\mathrm{arccosh}(2^{\frac{1}{n}})}$, we have $\sinh r\leq 2^{\frac{1}{n}} r$ for $r\leq r_s$  by Langrange's mean value theorem, which implies that 
	\begin{equation}\label{Eq-ULB3}
		|F_2^r|_g\leq b_n\sinh^n r\leq 2b_n r^n\leq 2b_n r_s^n=2b_n d_3^n\min\{1,m\}.
	\end{equation}
	If $m\geq 1$, then 
	\begin{align*}
	|F_2^r|_g&\leq 2b_n d_3^n\leq\varepsilon(1-2b_n d_3^n)<\varepsilon,\\
	|F_2^r|_g&\leq 2b_n d_3^n\leq\varepsilon(1-2b_n d_3^n)\leq \varepsilon(m-2b_n d_3^n)\leq \varepsilon|F_1^r|_g,
	\end{align*}
	which satisfies the condition \eqref{Lem-Cri2}.
		
  \noindent{If} $m<1$, then 
	\begin{align*}
	|F_2^r|_g&\leq 2b_n d_3^n m\leq\frac{\varepsilon}{1+\varepsilon}m<\frac{\varepsilon}{1+\varepsilon}<\varepsilon,\\
	|F_2^r|_g&\leq 2b_n d_3^n m\leq\frac{\varepsilon}{1+\varepsilon}m\leq\varepsilon |F_1^r|_g,
	\end{align*}
which also satisfies the condition \eqref{Lem-Cri2}. However, since $E$ is a minimizer, then \eqref{Lem-Cri1} cannot hold for any $r\leq r_s$, i.e., for $\forall r\leq r_s$, we have
\begin{equation}\label{Eq-ULB4}
	\Sigma^r:=P(F_1^r)+P(F_2^r)-P(E)>\frac{1}{2}\mathcal{E}(F_2^r)>\frac{1}{2}P(F_2^r).
\end{equation}
On the other hand, by \cite[Proposition 1]{Ambrosio01} and \cite[Theorem 3.61]{Ambrosio-Fusco-Pallara00}, we have
\begin{equation}\label{Eq-ULB5}
		\Sigma^r=2H^{n-1}(\partial^{*}F_1^r\cap\partial^{*}F_2^r),
\end{equation}
In fact, since all the points belonging to $\partial^{*}F_1^r$ and $\partial^{*}F_2^r$ are supported on $\partial B_r(x)$ and have density $\frac{1}{2}$ by \cite[Theorem 3.61]{Ambrosio-Fusco-Pallara00}, we have $\partial^{*}F_1^r\cap\partial^{*}F_2^r=\mathring{E}^M\cap\partial B_r(x)$. Combining this with \eqref{Eq-ULB4} and \eqref{Eq-ULB5}, we have
\begin{align*}
	2H^{n-1}(\mathring{E}^M\cap\partial B_r(x))>\frac{1}{2}\left[H^{n-1}(\partial^{*}E\cap B_r(x))+H^{n-1}(\mathring{E}^M\cap\partial B_r(x))\right], 
\end{align*}
which gives
\begin{equation}\label{Eq-ULB6}
	H^{n-1}(\mathring{E}^M\cap\partial B_r(x))>\frac{1}{3}H^{n-1}(\partial^{*}E\cap B_r(x)).
\end{equation}
Define
\begin{equation}\label{Eq-ULB7}
	r_t=\sup\left\{0\leq r\leq r_s:|E\cap B_r(x)|_g\geq\frac{1}{2}|B_r(x)|_g\right\},
\end{equation}
and we consider the following two cases.

(i) If $r_t=r_s<1$, then we have 
\begin{equation}\label{Eq-ULB8}
	|E\cap B_1(x)|_g\geq |E\cap B_{r_s}(x)|\geq \frac{1}{2}|B_{r_s}(x)|_g\geq\frac{b_n}{2}r_s^n=\frac{b_n d_3^n}{2}\min\{1,m\}.
\end{equation}

(ii)If $r_t<r_s$, then for $r\in (r_t,r_s]$, we have $|E\cap B_r(x)|_g<\frac{1}{2}|B_r(x)|_g$. Since $r_s<1$, then by the relative isoperimetric inequality \eqref{Rem-Eulik1}, we have
\begin{equation}\label{Eq-ULB9}
	P(E;B_r(x))=H^{n-1}(\partial^{*}E\cap B_r(x))\geq c_1|E\cap B_r(x)|_g^{\frac{n-1}{n}},
\end{equation}
where $c_1$ is some constant depending only on $n$. Denote $U(r)=|E\cap B_r(x)|_g$, then combining the coarea formula with \eqref{Eq-ULB6} and \eqref{Eq-ULB9}, we have for a.e. $r_t<r<r_s$, there holds
\begin{equation}\label{Eq-ULB10}
	\frac{d U(r)}{dr}\geq\frac{c_1}{3} U(r)^{\frac{n-1}{n}}.
\end{equation} 
Since $x\in\bar{E}^M$, we have $U(r)>0$ for all $r>0$ and hence solving the differential inequality \eqref{Eq-ULB10} gives
\begin{align*}
	U^{\frac{1}{n}}(r)&\geq U^{\frac{1}{n}}(r_t)+\frac{c_1}{3n}(r-r_t)\\
	&= \left(\frac{1}{2}|B_{r_t}(x)|_g\right)^{\frac{1}{n}}+\frac{c_1}{3n}(r-r_t)\\
	&\geq \left(\frac{b_n}{2}\right)^{\frac{1}{n}}r_t+\frac{c_1}{3n}(r-r_t)\\
	&\geq\min\left\{\left(\frac{b_n}{2}\right)^{\frac{1}{n}},\frac{c_1}{3n}\right\}r ,\,\,\forall r_t\leq r\leq r_s.
\end{align*}
In particular, we have
\begin{equation}\label{Eq-ULB11}
	|E\cap B_1(x)|_g=U(1)\geq U(r_s)\geq \min\left\{\frac{b_n}{2},\left(\frac{c_1}{3n}\right)^n\right\}d_3^n\min\{1,m\}.
\end{equation}
Then combining \eqref{Eq-ULB8} with \eqref{Eq-ULB11}, we prove this lemma with the constant
\begin{equation*}
	c_6=\min\left\{\frac{b_n}{2},\left(\frac{c_1}{3n}\right)^n\right\}d_3^n.
\end{equation*}
\end{proof}
\section{Existence of minimizers}\label{Sec-Existence}
In this section, we prove that the functional $\mathcal{E}(\cdot)$ admits a minimizer for small volumes. Inspired by \cite[Lemma 4.5]{Figalli15}, we firstly prove the following technical lemma.
\begin{lem}\label{Lem-Re-Pe}
	Let $n\geq 2$, $r_0>0$ and $E\subset\mathbb{H}^n$, if $|E\setminus B_{r_0}|_g\leq\eta$ for some $\eta>0$, then there exists $r_0\leq r_E\leq r_0+C_1(n)\eta^{\frac{1}{n}}$, such that
	\begin{equation}\label{Eq-Re-Pe-1}
		P(E\cap B_{r_E})\leq P(E)-\frac{|E\setminus B_{r_E}|_g}{C_2(n)\eta^{\frac{1}{n}}},
	\end{equation}
	where $C_1(n)=4 b_n^{-\frac{1}{n}}$ and $C_2(n)=2n^{-1}b_n^{-\frac{1}{n}}$.
\end{lem}
\begin{proof}
	Without loss of generality, we consider a set $E$ with $|E\setminus B_{r_0}|_g\leq\eta$ and $|E\setminus B_{r_0+C_1\eta^{\frac{1}{n}}}|_g>0$. Correspondingly, if we set $u(r)=|E\setminus B_r|_g$ with $r>0$, then $u$ is a decreasing function with 
	\begin{equation}\label{Eq-Re-Pe-2}
		\left\{\begin{aligned}
		&[0,r_0+C_1\eta^{\frac{1}{n}}]\subset \text{spt}\, u,\\
		&u(r_0)\leq\eta,\\
		&u'(r)=-H^{n-1}(\mathring{E}^M\cap\partial B_r)\,\, \text{for a.e.}\,\, r>0.
		\end{aligned}\right.
	\end{equation}
	Arguing by contradiction, we now assume that 
	\begin{equation}\label{Eq-Re-Pe-3}
		P(E)\leq P(E\cap B_{r})+\frac{u(r)}{C_2\eta^{\frac{1}{n}}},\quad \forall r\in[r_0,r_0+C_1\eta^{\frac{1}{n}}].
	\end{equation}
We decompose $E$ as $E=(E\cap B_r)\cup (E\setminus B_r)$, then we have 
\begin{equation}\label{Eq-Re-Pe-4}
	P(E)=P(E\cap B_r)+P(E\setminus B_r)-2H^{n-1}(\mathring{E}^M\cap\partial B_r).
\end{equation}
Combining \eqref{Eq-HEin} (taking $F=E\setminus B_r$) with \eqref{Eq-Re-Pe-2}-\eqref{Eq-Re-Pe-4} and the fact that $u(r)\leq\eta$, we see that for a.e. $r\in[r_0,r_0+C_1\eta^{\frac{1}{n}}]$, there holds
\begin{equation}
  u'(r)\leq-\frac{nb_n^{\frac{1}{n}}}{4}u(r)^{\frac{n-1}{n}},
\end{equation}
which gives
\begin{equation}
	u^{\frac{1}{n}}(r_0+C_1\eta^{\frac{1}{n}})\leq u^{\frac{1}{n}}(r_0)-\frac{b_n^{\frac{1}{n}}}{4}C_1\eta^{\frac{1}{n}}\leq 0,
\end{equation}
where we have used the fact that $u(r_0)\leq\eta$ and the definition of $C_1$ in the last inequality. It contradicts the assumption that $u(r_0+C_1\eta^{\frac{1}{n}})>0$.
\end{proof}

\begin{thm}\label{Lem-Min}
	There exists $m_0>0$ depending on $n,\alpha$ and $\gamma$, such that for every $m\leq m_0$ and every set of finite perimeter $F$ with $|F|_g=|B_r|_g=m$, there exists a set of finite perimeter $G$ and a positive constant $C_3$ depending only on $n$, $\alpha$ and $\gamma$, such that $G$ is contained in a geodesic ball of radius $(1+C_3 r^{\frac{1}{2n}})r$, which satsifies $|G|_g=m$ and $\mathcal{E}(G)\leq\mathcal{E}(F)$. In particular, for $m\leq m_0$, the minimization problem \eqref{Prob-mini} admits a minimizer $E$ with volume $m$ which is contained in a geodesic ball of radius $(1+C_3 r^{\frac{1}{2n}})r$.
\end{thm}
\begin{proof}
Firstly, we prove the existence of such set $G$. In the following of the proof, we always assume that $m\leq |B_1|_g$. Then directly we have $r\leq 1$ and hence $\sinh r\leq\cosh 1\cdot r$.  Combined with \eqref{Eq-ULB1}, we have
\begin{equation}\label{Eq-Min1}
	b_n r^n \leq m=|B_r|_g\leq b_n\sinh^n r\leq b_n \cosh^n 1\cdot r^n
\end{equation}
Take a point $p\in\mathbb{H}^n$, if $\mathcal{E}(F)>\mathcal{E}(B_r(p))$, then we can choose  $G=B_r(p)$, which satisfies the requirements of this theorem. Hence in the following, we may assume that
\begin{align}
 \mathcal{E}(F)&\leq\mathcal{E}(B_r(p))\notag\\
 &\leq nb_n\sinh^{n-1} r+c_3\gamma m\notag\\
 &\leq nb_n\cosh^{n-1}1\cdot r^{n-1}+\frac{n b_n^{\frac{n+\alpha}{n}}\gamma}{n-\alpha}\cosh^n 1\cosh^{n-1}(b_n^{-\frac{1}{n}}|B_1|_g^{\frac{1}{n}})|B_1|_g^{\frac{n-\alpha}{n}}\cdot r^n\notag\\
 &\leq C_4(r^{n-1}+r^n), \label{Eq-Min0}
\end{align}
where we have used Lemma \ref{Lem-Bound} with $\bar{m}=|B_1|_g$ and \eqref{Eq-Min1}. Here the constant $C_4$ can be chosen as
\begin{equation}\label{Defn-C_4}
	C_4=nb_n\cosh^{n-1}1+\frac{n b_n^{\frac{n+\alpha}{n}}\gamma}{n-\alpha}\cosh^n 1\cosh^{n-1}(b_n^{-\frac{1}{n}}|B_1|_g^{\frac{1}{n}})|B_1|_g^{\frac{n-\alpha}{n}},
\end{equation}
 is a constant depending on $n,\alpha$ and $\gamma$.
	Under this assumption, according to Theorem \ref{Thm-quan}, there exists a constant $c'$ depending only on $n$ (since $r\leq 1$), such that 
	\begin{equation}\label{Eq-Min2}
		P(F)-P(B_r)\geq \frac{c'}{\sinh^{n+1} r}\bm{\upalpha}^2(F)\geq \frac{c'}{\cosh^{n+1}1\cdot r^{n+1}}\bm{\upalpha}^2(F),
	\end{equation}
	where $\bm{\upalpha}(F)$ is hyperbolic Fraenkel asymmetry index defined in \eqref{Defn-asym}. 	Also, since $\mathcal{E}(F)\leq\mathcal{E}(B_r)$, we have
	\begin{align}
		P(F)-P(B_r)&\leq \gamma\left(NL_{\alpha}(B_r)-NL_{\alpha}(F)\right)<\gamma NL_{\alpha}(B_r)\notag\\
		&\leq c_3\gamma m\leq \frac{n b_n^{\frac{n+\alpha}{n}}\gamma}{n-\alpha}\cosh^n 1\cosh^{n-1}(b_n^{-\frac{1}{n}}|B_1|_g^{\frac{1}{n}})|B_1|_g^{\frac{n-\alpha}{n}}\cdot r^n,\label{Eq-Min3}
	\end{align}
	where we have used Lemma \ref{Lem-Bound} with $\bar{m}=|B_1|_g$ in the third inequality and \eqref{Eq-Min1} in the last inequality.  Combining \eqref{Eq-Min2} with \eqref{Eq-Min3}, we have for some constant $d_4$ depending on $n,\alpha$ and $\gamma$, such that
	\begin{equation}\label{Eq-Min4}
  \bm{\upalpha}(F)\leq d_4 r^{n+\frac{1}{2}}.
	\end{equation}
	Namely, after a suitable hyperbolic isometry of $F$ (and we also denote the new set as $F$), we have $|F\Delta B_r(p)|_g\leq d_4 r^{n+\frac{1}{2}}$. Since $|F|_g=|B_r(p)|_g=m$, we have $|F\Delta B_r(p)|_g=2|F\setminus B_r(p)|_g$ and hence
	\begin{equation}\label{Eq-Min5}
		|F\setminus B_r(p)|_g\leq\frac{1}{2}d_4 r^{n+\frac{1}{2}}:=\eta,
	\end{equation}
	then by Lemma \ref{Lem-Re-Pe}, we see that there exists $r\leq r_{F}\leq r+C_1\eta^{\frac{1}{n}}$, such that 
	\begin{equation*}
		P(F\cap B_{r_F}(p))\leq P(F)-\frac{|F\setminus B_{r_F}(p)|_g}{C_2\eta^{\frac{1}{n}}}.
	\end{equation*}
Combined this with the fact that $NL_{\alpha}(F\cap B_{r_F}(p))\leq NL_{\alpha}(F)$, we have
\begin{equation}\label{Eq-Min6}
		\mathcal{E}(F\cap B_{r_F}(p))\leq \mathcal{E}(F)-\frac{|F\setminus B_{r_F}(p)|_g}{C_2\eta^{\frac{1}{n}}}.
\end{equation}
Denote 
\begin{equation}\label{Eq-Min7}
	u=\frac{|F\setminus B_{r_F}(p)|_g}{m}\leq\frac{d_4 r^{\frac{1}{2}}}{2b_n},
\end{equation}
then $|F\cap B_{r_F}(p)|_g=m(1-u)$ and hence if $r\leq r_{h}:=\frac{b_n^2}{d_4^2}$, then the function $u\in[0,\frac{1}{2}]$. Denote $G=\Phi_{\lambda}(F\cap B_{r_F}(p))$ such that $|G|_g=|F|_g=m$, then by Proposition \ref{Prop-VP}, we have $\lambda=(1-u)^{\frac{1}{n-1}}<1$ and there holds 
\begin{align}
	P(G)&\leq (1-u)^{-2} P(F\cap B_{r_F}(p)),\label{Eq-Min8}\\
	NL_{\alpha}(G)&\leq (1-u)^{-2}NL_{\alpha}(F\cap B_{r_F}(p))\label{Eq-Min9}
\end{align}
according to Proposition \ref{Prop-Lip} and \ref{Prop-VP}. Taking into accout of the following basic inequality (cf. \cite[Page 469]{Figalli15}):
\begin{equation}\label{Eq-Min10}
	(1-u)^{-p}\leq 1+2^{p+1}u,\,\,p\in\mathbb{R}\,\,\text{and}\,\,u\in[0,\frac{1}{2}],
\end{equation}
we have that for $r\leq\min\{1,\frac{b_n^2}{d_4^2}\}$,
\begin{align}
	\mathcal{E}(G)&\leq (1+8u)\mathcal{E}(F\cap B_{r_F}(p))\notag\\
	&\leq\mathcal{E}(F)-\frac{mu}{C_2\eta^{\frac{1}{n}}}+8C_4 u(r^{n-1}+r^n)\notag\\
	&\leq\mathcal{E}(F)-u r^{n-1-\frac{1}{2n}}\bigg[\frac{2^{\frac{1}{n}}b_n}{C_2 d_4^{\frac{1}{n}}}-8C_4 r^{\frac{1}{2n}}(1+r)\bigg].\label{Eq-Min11}
\end{align}
Then there exists a radius $r_{h'}$ depending on $n,\alpha$ and $\gamma$, such that if $r\leq r_{h'}$, the inequality in \eqref{Eq-Min11} gives $\mathcal{E}(G)\leq\mathcal{E}(F)$. Since for any $x\in F\cap B_{r_F}(p)$, 
\begin{equation*}
	d_g(x,p)\leq r_F\leq r+C_1\eta^{\frac{1}{n}}=(1+2^{-\frac{1}{n}}d_4^{\frac{1}{n}}C_1 r^{\frac{1}{2n}})r,
\end{equation*}
then combined with Proposition \ref{Prop-Lip} and \eqref{Eq-Min10}, for $\forall y=\Phi_{\lambda}(x)\in G$,
\begin{align}
	d_g(y,\Phi_{\lambda}(p))&\leq (1-u)^{-\frac{2}{n-1}} d_g(x,p)\notag\\
	&\leq (1+2^{\frac{n+1}{n-1}}u)(1+2^{-\frac{1}{n}}d_4^{\frac{1}{n}}C_1 r^{\frac{1}{2n}})r\notag\\
	&\leq (1+2^{\frac{2}{n-1}}b_n^{-1}d_4 r^{\frac{1}{2}})(1+2^{-\frac{1}{n}}d_4^{\frac{1}{n}}C_1 r^{\frac{1}{2n}})r\notag\\
	&\leq (1+C_3 r^{\frac{1}{2n}})r,
\end{align}
for some constant $C_3>0$ since we have already assumed that $r\leq 1$.
 Hence if we define 
\begin{equation}\label{Defn-m0}
	m_0=\min\{|B_1|_g,|B_{r_h}|_g,|B_{r_{h'}}|_g\},
\end{equation}
then for $m\leq m_0$, the set $G$ constructed above in both cases and the constant $C_3$ satsify the requirements of this theorem. 

Then we are left to prove that for $m\leq m_0$, the minimization problem \eqref{Prob-mini} admits a minimizer $E$ with volume $m$ which is contained in a geodesic ball of radius $(1+C_3 r^{\frac{1}{2n}})r$. Firstly, we can choose a sequence of sets of finite perimeter $\{F_k\}$ with $|F_k|_g=m$ such that 
\begin{equation*}
	\lim_{k\to\infty}\mathcal{E}(F_k)=\inf_{|F|_g=m}{\mathcal{E}(F)}.
\end{equation*}
By the argument above, we can further choose a minimizing sequence $\{G_k\}$, such that up to isometries, $G_k\subset B_{(1+C_3 r^{\frac{1}{2n}})r}(p)$ for some fixed $p\in\mathbb{H}^n$. Then by the lower semicontinuity of perimeter (cf. \cite[Proposition 12.15]{Maggi-12}) and the compactness of perimeter (cf. \cite[Theorem 12.26]{Maggi-12}), there exists a set of finite perimeter $E\subset\mathbb{H}^n$, a fixed point $p\in\mathbb{H}^n$ and a subsequence, which we also denote as $\{G_k\}$ such that
\begin{align}
	G_k\to& E,\quad E\subset B_{(1+C_3 r^{\frac{1}{2n}})r}(p),\\
	P(E)&\leq\liminf_{k\to\infty} P(G_k).\label{Eq-peri}
\end{align}
Since $G_k\to E$, we have $m=|G_k|_g\to|E|_g$ and hence $|E|_g=m$. Furthermore, by Lemma \ref{Lem-NL}, we have
\begin{equation}\label{Eq-NLa}
	\lim_{k\to\infty}{NL_{\alpha}(G_k)}=NL_{\alpha}(E).
\end{equation}
Combining \eqref{Eq-peri} with \eqref{Eq-NLa} gives
\begin{align*}
	\inf_{|F|_g=m}{\mathcal{E}(F)}\leq\mathcal{E}(E)&=P(E)+\gamma NL_{\alpha}(E)\\
	&\leq \liminf_{k\to\infty} P(G_k)+\gamma\lim_{k\to\infty}{NL_{\alpha}(G_k)}\\
	&=\liminf_{k\to\infty}\mathcal{E}(G_k)\\
	&\leq \liminf_{k\to\infty}\mathcal{E}(F_k)=\inf_{|F|_g=m}{\mathcal{E}(F)},
\end{align*}
which implies that $E$ is a minimizer of the functional $\mathcal{E}(\cdot)$ with volume $m$. 
\end{proof}
\section{Geodesic ball as the minimizer}\label{Sec-Ball}
In this section, we give the proof of Theorem \ref{Thm-Existence}. By Proposition \ref{Prop-Quasi}, we know that for any given $r_0>0$, the minimizer to the minimization problem \eqref{Prob-mini} is an $(\omega,r_0)$-minimizer for the perimeter with some constant $\omega>0$. Using Theorem \ref{Lem-Min}, we can calculate which quantities $\omega$ precisely depends on provided $m\leq m_0$.
\begin{prop}\label{Prop-Qp}
	Let $n\geq 2$, $\alpha\in(0,n)$ and $\gamma>0$. Assume that $E$ is a minimizer with $|E|_g=|B_r|_g=m$ for $m\leq m_0$, then there exist two constants $\Lambda_1$ and $\Lambda_2$ depending only on $n$, $\alpha$ and $\gamma$ such that
\begin{equation}\label{In-Qp}
	P(E)\leq P(F)+\left(\frac{\Lambda_1}{r}+\Lambda_2\right)|E\Delta F|_g
\end{equation}
for every $F\subset\mathbb{H}^n$. Here $\Lambda_1$ and $\Lambda_2$ can be chosen as 
\begin{align}
	\Lambda_1&=\frac{6C_4}{b_n},\label{Defin-Lamb1}\\
	\Lambda_2&=\max\left\{\frac{6C_4}{b_n}+14\gamma c_3,\gamma c_3\left(2C_5^{\frac{2(\alpha+1-n)}{n-1}}+C_5+1\right)\right\},\label{Defin-Lamb2}
\end{align}
where $C_4$ is the constant defined in \eqref{Defn-C_4},
\begin{equation}\label{Defn-C5}
	C_5=\left(\frac{2C_4}{nb_n}\right)^{\frac{n}{n-1}},
\end{equation}
and $c_3$ is the constant defined in Lemma \ref{Lem-Bound} depending on $n$, $\alpha$ and $C_5 m_0$ (we choose $\bar{m}=C_5 m_0$).
\end{prop}
\begin{proof}
	In the following of proof, we always assume that $P(F)\leq P(E)$, since otherwise \eqref{In-Qp} holds trivially. We claim that we can reduce to prove \eqref{In-Qp} in the case that 
	\begin{equation}\label{In-twos}
		\frac{1}{2}\leq\frac{|F|_g}{m}\leq C_5.
	\end{equation}
	Indeed, if we compare $E$ with $B_r$, by \eqref{Eq-Min0}, we have
	\begin{equation}\label{Qp1}
		\mathcal{E}(E)=P(E)+\gamma NL_{\alpha}(E)\leq C_4(r^{n-1}+r^n).
	\end{equation}
	Then if $|F|_g<\frac{1}{2}m$, we have
	\begin{equation*}
		|E\Delta F|_g\geq |E|_g-|F|_g>\frac{1}{2}m\geq\frac{b_n}{2} r^n.
	\end{equation*}
	Combined with the definition of $\Lambda_1$ and $\Lambda_2$ gives \eqref{In-Qp}. Also, if $|F|_g>C_5 m\geq C_5 b_n r^n$, then by \eqref{Eq-HEin} and $m\leq m_0$ (which implies $r\leq 1$ by \eqref{Defn-m0}), we have
	\begin{equation*}
		P(F)\geq nb_n^{\frac{1}{n}}|F|_g^{\frac{n-1}{n}}\geq nb_n C_5^{\frac{n-1}{n}}r^{n-1}=2C_4 r^{n-1}\geq P(E),
	\end{equation*}
	which also gives \eqref{In-Qp}. 
    
    We would like to emphasize here that in the subsequent proof of this proposition, we will frequently make use of Lemma \ref{Lem-Bound} and Lemma \ref{Lem-NL}. It should be noted that the constant $c_3$ in these two lemmas depends on $n$, $\alpha$ and the upper bound of the volume $\bar{m}$. By the previous argument, $\bar{m}$ can be taken as $C_5 m_0$, which is a quantity that only depends on $n$, $\alpha$ and $\gamma$. Thus, $c_3$ is also a constant that only depends on $n$, $\alpha$ and $\gamma$. Let $D:=\Phi_{\lambda}(F)$ be such that $|D|_g=\lambda^{1-n}|F|_g=m$, which implies
	\begin{equation*}
		\lambda=\left(\frac{|F|_g}{m}\right)^{\frac{1}{n-1}}.
	\end{equation*}
We consider the following two cases.

$\bullet\,\,\bf{Case\,1.}$ $2^{-\frac{1}{n-1}}\leq\lambda\leq 1$. Arguing as in \eqref{Eq-Min8} and \eqref{Eq-Min9}, we have
\begin{align}
	P(E)&\leq P(D)+\gamma NL_{\alpha}(D)-\gamma NL_{\alpha}(E)\notag\\
	&\leq {\lambda}^{2-2n}(P(F)+\gamma NL_{\alpha}(F))-\gamma NL_{\alpha}(E)\notag\\
	&={\lambda}^{2-2n} P(F)+{\lambda}^{2-2n}\gamma(NL_{\alpha}(F)-NL_{\alpha}(E))+\gamma({\lambda^{2-2n}}-1)NL_{\alpha}(E).\label{Qp2}
\end{align}
Since $\lambda^{2-2n}\leq 4$ by assumption and also
\begin{align*}
	{\lambda}^{2-2n}&=\left(\frac{m}{|F|_g}\right)^2=\left(1+\frac{m-|F|_g}{|F|_g}\right)^2\\
	&\leq 1+3\frac{m-|F|_g}{|F|_g}\leq 1+6 \frac{|E\Delta F|_g}{m}.
\end{align*}
Hence 
\begin{align*}
	P(E)&\leq\left(1+6 \frac{|E\Delta F|_g}{m}\right)P(F)+8\gamma c_3|E\Delta F|_g+6\gamma\frac{|E\Delta F|_g}{m}NL_{\alpha}(E)\\
	&\leq P(F)+\frac{6C_4}{m}(r^{n-1}+r^n)|E\Delta F|_g+14\gamma c_3|E\Delta F|_g\\
	&\leq P(F)+\left(\frac{6C_4}{b_n r}+\frac{6C_4}{b_n}+14\gamma c_3\right)|E\Delta F|_g\\
	&\leq P(F)+\left(\frac{\Lambda_1}{r}+\Lambda_2\right)|E\Delta F|_g,
\end{align*}
where we have used the assumption $P(F)\leq P(E)\leq C_4(r^{n-1}+r^n)$ in the second inequality.

$\bullet\,\,\bf{Case\,2.}$ $1\leq\lambda\leq C_5^{\frac{1}{n-1}}$. Arguing as in \eqref{Eq-Quasi3} and \eqref{Eq-Quasi4}, we have
\begin{align}
	P(E)&\leq P(D)+\gamma NL_{\alpha}(D)-\gamma NL_{\alpha}(E)\notag\\
	&\leq P(F)+\gamma{\lambda}^{-2(n-\alpha-1)}NL_{\alpha}(F)-\gamma NL_{\alpha}(E)\notag\\
	&= P(F)+\gamma{\lambda}^{-2(n-\alpha-1)}(NL_{\alpha}(F)-NL_{\alpha}(E))+\gamma({\lambda}^{-2(n-\alpha-1)}-1)NL_{\alpha}(E).\label{Qp3}
\end{align}
If $0<\alpha\leq n-1$, i.e., $\lambda^{-2(n-\alpha-1)}\leq 1$, then \eqref{Qp3} gives
\begin{align*}
	P(E)&\leq P(F)+\gamma|NL_{\alpha}(F)-NL_{\alpha}(E)|\\
	&\leq P(F)+2\gamma c_3|E\Delta F|_g\\
	&\leq P(F)+\left(\frac{\Lambda_1}{r}+\Lambda_2\right)|E\Delta F|_g.
\end{align*}
If $n-1<\alpha<n$, i.e. $1\leq\lambda^{-2(n-\alpha-1)}=\lambda^{2(\alpha+1-n)}\leq C_5^{\frac{2(\alpha+1-n)}{n-1}}$, and also
\begin{align*}
	\lambda^{2(\alpha+1-n)}&=\left(\frac{|F|_g}{m}\right)^{\frac{2(\alpha+1-n)}{n-1}}=\left(1+\frac{|F|_g-m}{m}\right)^{\frac{2(\alpha+1-n)}{n-1}}\\
	&\leq \left(1+\frac{|F|_g-m}{m}\right)^2\leq 1+(C_5+1)\frac{|F|_g-m}{m}\leq 1+(C_5+1)\frac{|E\Delta F|_g}{m}.
\end{align*}
Combined with \eqref{Qp3} gives
\begin{align*}
	P(E)&\leq P(F)+2\gamma c_3 C_5^{\frac{2(\alpha+1-n)}{n-1}}|E\Delta F|_g+\gamma c_3(C_5+1)|E\Delta F|_g\\
	&=P(F)+\gamma c_3\left(2C_5^{\frac{2(\alpha+1-n)}{n-1}}+C_5+1\right)|E\Delta F|_g\\
	&\leq P(F)+\left(\frac{\Lambda_1}{r}+\Lambda_2\right)|E\Delta F|_g.
\end{align*}
This completes the proof of this proposition.
\end{proof}
We then show that for $m\leq m_0$, the minimizer to the minimization problem \eqref{Prob-mini} is indeed a perimeter $(K,r_1)$-quasiminimizer for some suitable $K$ and $r_1$. 
\begin{lem}\label{Lem-Qua}
	For $m\leq m_0$, the minimizer $E$ to the minimization problem \eqref{Prob-mini} with $|E|_g=|B_r|_g=m$ is a perimeter $\left(3,\frac{r}{C_6}\right)$-quasiminimizer for the constant $C_6=\max\left\{1, \frac{2(\Lambda_1+\Lambda_2)}{n-1}\cosh 1\right\}$.
\end{lem}
\begin{proof}
	By Theorem \ref{Lem-Min}, we see that for $m\leq m_0$, the minimizer $E$ satisfies $E\subset B_{(1+C_3 r^{\frac{1}{2n}})r}(p)$ for some point $p$. Then for any geodesic ball $B_s(x)$ for $0<s\leq \frac{r}{C_6}$, if $B_s(x)\cap B_{(1+C_3 r^{\frac{1}{2n}})r}(p)=\emptyset$, then this lemma trivially holds since $P(E;B_s(x))=0$. Hence in the following, we may assume that $B_s(x)\cap B_{(1+C_3 r^{\frac{1}{2n}})r}(p)\neq\emptyset$. For any set $F\subset\mathbb{H}^n$ such that $E\Delta F\subset\subset B_s(x)$ with $s\leq\frac{r}{C_6}$, we have
	\begin{align}
		P(E;B_s(x))-P(F;B_s(x))&=P(E)-P(F)\notag\\
		&\leq\left(\frac{\Lambda_1}{r}+\Lambda_2\right)|E\Delta F|_g.\label{Eq-lem-Qua-1}
	\end{align}
By \cite[Lemma 2.2]{Bogelein15}, we have
\begin{align}
	|E\Delta F|_g&\leq\frac{\sinh(\frac{r}{C_6})\cosh(\frac{r}{C_6})}{n\cosh^2(\frac{r}{C_6})-\sinh^2(\frac{r}{C_6})}P(E\Delta F)\notag\\
	&\leq\frac{\sinh(\frac{r}{C_6})}{(n-1)\cosh(\frac{r}{C_6})}P(E\Delta F)\notag\\
	&\leq \frac{\sinh(\frac{r}{C_6})}{n-1}\left(P(E;B_s(x))+P(F;B_s(x))\right).\label{Eq-lem-Qua-2}
\end{align}
Combining \eqref{Eq-lem-Qua-1} with \eqref{Eq-lem-Qua-2} and noting that $r\leq 1$, we arrive at
\begin{align*}
	P(E;B_s(x))-P(F;B_s(x))&\leq \frac{\Lambda_1+\Lambda_2}{n-1}\frac{\sinh(\frac{r}{C_6})}{r}\left(P(E;B_s(x))+P(F;B_s(x))\right)\\
	&\leq \frac{\Lambda_1+\Lambda_2}{n-1}\frac{\cosh 1}{C_6}\left(P(E;B_s(x))+P(F;B_s(x))\right)\\
	&\leq \frac{1}{2}\left(P(E;B_s(x))+P(F;B_s(x))\right),
\end{align*}
which gives $P(E;B_s(x))\leq 3 P(F;B_s(x))$ and completes the proof of this lemma.
\end{proof}
 {\bf{Proof of Theorem \ref{Thm-Existence}}}. With the preparations above, we are now ready to prove Theorem  \ref{Thm-Existence} for some constant $m_1\leq m_0$. Argue by contradiction, if $m_1$ does not exist, then there exist two sequences $\{E_k\}\subset\mathbb{H}^n$ (which are all not geodesic balls) and $\{x_k\}\in\mathbb{H}^n$ for $k\geq 1$, such that
 \begin{equation}
  m_k=|E_k|_g=|B_{r_k}|_g\to 0,\,\, E_k\subset B_{(1+C_3 r_k^{\frac{1}{2n}})r_k}(x_k).
 \end{equation}
 We first show that $\forall \delta\in(0,1)$, there exists $k_1=k_1(\delta)$, such that 
 \begin{equation}
 	B_{(1-\delta)r_k}(x_k)\subset E_k\subset  B_{(1+C_3 r_k^{\frac{1}{2n}})r_k}(x_k)
 \end{equation}
 for $k\geq k_1$. Otherwise, if there exists some $\delta_0\in(0,1)$ and infinitely many $k$, we can find some point $\omega_k\in\partial E_k\cap B_{(1-\delta_0)r_k}(x_k)$. Then by Theorem \ref{Thm-quasim} and Lemma \ref{Lem-Qua}, there exists $y_k\in B_{\frac{\delta_0 r_k}{2C_6}}(\omega_k)$, such that $B_{\frac{\delta_0 r_k}{2C_0C_6}}(y_k)\subset\mathbb{H}^n\setminus E_k$ for some $C_0>1$ depending only on $n$. Since 
\begin{equation*}
	(1-\delta_0)r_k+\frac{\delta_0 r_k}{2C_6}+\frac{\delta_0 r_k}{2C_0C_6}<r_k,
\end{equation*}
we have 
\begin{equation*}
	B_{\frac{\delta_0 r_k}{2C_0C_6}}(y_k)\subset B_{r_k}(x_k)\setminus {E_k}\subset B_{(1+C_3 r_k^{\frac{1}{2n}})r_k}(x_k)\setminus E_k,
\end{equation*}
which gives
\begin{equation}\label{Eq-loww}
	\frac{\left|B_{(1+C_3 r_k^{\frac{1}{2n}})r_k}(x_k)\Delta E_k\right|_g}{|E_k|_g}\geq \frac{\left|B_{\frac{\delta_0 r_k}{2C_0C_6}}(y_k)\right|_g}{|E_k|_g}\geq\left(\frac{\delta_0}{2C_0C_6\cosh 1}\right)^n>0.
\end{equation}
On the other hand, 
\begin{equation*}
	\lim_{k\to\infty}\frac{\left|B_{(1+C_3 r_k^{\frac{1}{2n}})r_k}(x_k)\Delta E_k\right|_g}{|E_k|_g}=\lim_{r_k\to 0}\frac{\left|B_{(1+C_3 r_k^{\frac{1}{2n}})r_k}\right|_g-\left|B_{r_k}\right|_g}{|B_{r_k}|_g}=0,
\end{equation*}
which contradicts with \eqref{Eq-loww}. Next, for $\forall\delta\in(0,1)$, we can take $k_2=k_2(\delta)$ such that for $k\geq k_2$, $C_3 r_k^{\frac{1}{2n}}\leq\delta$. Hence, for $k\geq\max\{k_1,k_2\}$, we have
\begin{equation}\label{Eq-incl}
	B_{(1-\delta)r_k}(x_k)\subset E_k\subset  B_{(1+\delta)r_k}(x_k),
\end{equation}
from which it is easy to find a sequence $\delta_k$, non-increasing and tending to $0$, such that for any $k\in\mathbb{N}$,
\begin{equation}\label{Eq-delk}
		B_{(1-\delta_k)r_k}(x_k)\subset E_k\subset  B_{(1+\delta_k)r_k}(x_k).
\end{equation}
Combining \eqref{Eq-delk} with \cite[Lemma 2.6]{Bogelein15}, we see that the barycenter $p_k$ of $E_k$ satisfies
\begin{equation*}
	\frac{d_g(x_k,p_k)}{r_k}\to 0
\end{equation*}
as $k\to\infty$. Then for $\forall\delta\in(0,1)$, there exists $k_3:=k_3(\delta)$, such that for $k\geq k_3$, $d_g(x_k,p_k)\leq \delta r_k$. Combining this with \eqref{Eq-incl}, we conclude that for $\forall \delta\in(0,1)$, there exists $k_0:=\max\{k_1,k_2,k_3\}$, such that for $k\geq k_0$, we have
\begin{equation*}
	B_{(1-2\delta)r_k}(p_k)\subset E_k\subset  B_{(1+2\delta)r_k}(p_k).
\end{equation*}
By suitable hyperbolic isometries, we may assume that 
\begin{equation}\label{Eq-twi2}
	B_{(1-2\delta)r_k}(O)\subset E_k\subset  B_{(1+2\delta)r_k}(O),
\end{equation}
where $O$ is the origin of $\mathbb{H}^n$ and is the barycenter of all $E_k$. Then by a similar argument in \cite[\S 6.6.2]{Bogelein15}, we deduce that there exists a sequence of functions $u_k\in C^{1,\frac{1}{2}}(\mathbb{S}^n)$ which converges to 0 in $C^{1,\zeta}(\mathbb{S}^{n-1})$ for any $\zeta\in(0,\frac{1}{2})$, such that $\partial E_k$ is a radial graph with respect to $O$, in the sense that  
\begin{equation}
		\partial E_k=\{(x,r_k(1+u_k(x))):x\in\mathbb{S}^{n-1}\}.
\end{equation}
Since we assume that $E_k$ are not geodesic balls, hence we have
\begin{equation}\label{Eq-Q0}
	P(E_k)-P(B_{r_k})\leq \gamma(NL_{\alpha}(B_{r_k})-NL_{\alpha}(E_k)).
\end{equation}
By Theorem \ref{Thm-Fuglede}, Remark \ref{Rem-R1} and Remark \ref{Rem-Fuglede}, there holds
\begin{equation*}
\sinh r_k\geq\left(\frac{nb_n L_2}{\gamma L_1}\right)^{\frac{1}{n+1-\alpha}},
\end{equation*}
which is a contradiction if $k$ is sufficiently large. This completes the proof of Theorem \ref{Thm-Existence}.
\section{Nonexistence of minimizers}\label{Sec-Nonex}
In this section, we prove that the functional $\mathcal{E}(\cdot)$ does not admit minimizers for large volumes. Firstly, we prove an  interpolation estimate which was firstly derived in \cite[Lemma 7.1]{Knupfer-Muratov14} in the case of $\mathbb{R}^n$, but also holds in $\mathbb{H}^n$ through some minor modifications in its proof.
\begin{lem}\label{Lem-Inter}
	For any nonnegative $u\in BV(\mathbb{H}^n)\cap L^{\infty}(\mathbb{H}^n)$, there exists a constant $c_7$ depending only on $n$ and $\alpha$, such that 
	\begin{equation}\label{Lem-Inter1}
		\int_{\mathbb{H}^n}{u^2}\,dV_g\leq c_7||u||^{\frac{n-\alpha}{n+1-\alpha}}_{L^{\infty}(\mathbb{H}^n)}\left(\int_{\mathbb{H}^n}{|\nabla u|_g}\,dV_g\right)^{\frac{n-\alpha}{n+1-\alpha}}\left(\int_{\mathbb{H}^n}\int_{\mathbb{H}^n}\frac{u(x)u(y)}{d_g(x,y)^{\alpha}}\,dV_g(x)\,dV_g(y)\right)^{\frac{1}{n+1-\alpha}}.
	\end{equation}
\end{lem}
\begin{proof}
	Take $R>0$ to be determined, we have
\begin{align}
	&\int_{\mathbb{H}^n}\int_{\mathbb{H}^n}\frac{u(x)u(y)}{d_g(x,y)^{\alpha}}\,dV_g(x)\,dV_g(y)\notag\\
	\geq&\int_{\mathbb{H}^n}\int_{B_{2R}(x)\setminus B_R(x)}\frac{u(x)[u(x)+u(y)-u(x)]}{d_g(x,y)^{\alpha}}\,dV_g(y)\,dV_g(x)\notag\\
	=&\int_{\mathbb{H}^n}\int_{B_{2R}(x)\setminus B_R(x)}\frac{u^2(x)}{d_g(x,y)^{\alpha}}\,dV_g(y)\,dV_g(x)+\int_{\mathbb{H}^n}\int_{B_{2R}(x)\setminus B_R(x)}\frac{u(x)[u(y)-u(x)]}{d_g(x,y)^{\alpha}}\,dV_g(y)\,dV_g(x)\notag\\
	:=&I+II.\label{Eq-Inter1}
\end{align}
Term $I$ can be estimated as follows.
\begin{align}
	I&=\int_{\mathbb{H}^n}\int_{B_{2R}(x)\setminus B_R(x)}\frac{u^2(x)}{d_g(x,y)^{\alpha}}\,dV_g(y)\,dV_g(x)\notag\\
	&\geq (2R)^{-\alpha}|B_{2R}\setminus B_R|_g\int_{\mathbb{H}^n}{u^2(x)}\,dV_g(x)\notag\\
	&=(2R)^{-\alpha}\int_{\mathbb{S}^{n-1}}\int_R^{2R}\sinh^{n-1}r\,dr\,d\sigma\int_{\mathbb{H}^n}{u^2(x)}\,dV_g(x)\notag\\
	&=n b_n(2R)^{-\alpha}\int_R^{2R}\sinh^{n-1}r\,dr\int_{\mathbb{H}^n}{u^2(x)}\,dV_g(x).\label{Eq-Inter2}
\end{align}
Next we estimate term $II$. 
Given $x\in\mathbb{H}^n$ and $y\in B_{2R}(x)\setminus B_R(x)$, denote $\gamma_{x,y}(t)$ as the geodesic connecting $x$ and $y$ with $\gamma_{x,y}(0)=x$ and $\gamma_{x,y}(1)=y$. Then we have $|\gamma_{x,y}'(t)|_g\equiv d_g(x,y)$ for $t\in[0,1]$ and
\begin{align*}
	|u(y)-u(x)|&=|u(\gamma_{x,y}(1))-u(\gamma_{x,y}(0))|\\
	&=\bigg|\int_{0}^1{(u\circ\gamma_{x,y})'(t)}\,dt\bigg|\\
	&=\bigg|\int_{0}^1{\nabla u(\gamma_{x,y}(t))\cdot \gamma_{x,y}'(t)}\,dt\bigg|\\
	&\leq \int_{0}^1{\big|\nabla u(\gamma_{x,y}(t))\big|_g\cdot \big|\gamma_{x,y}'(t)}\big|_g\,dt\\
	&=d_g(x,y) \int_{0}^1{\big|\nabla u(\gamma_{x,y}(t))\big|_g}\,dt,
\end{align*}
which implies that for any $y\in B_{2R}(x)\setminus B_R(x)$,
\begin{align}
	\bigg|\frac{u(x)[u(y)-u(x)]}{d_g(x,y)^{\alpha}}\bigg|&\leq R^{-\alpha}|u(x)||u(y)-u(x)|\notag\\
	&\leq 2R^{1-\alpha}||u||_{L^{\infty}(\mathbb{H}^n)}\int_{0}^1{\big|\nabla u(\gamma_{x,y}(t))\big|_g}\,dt\label{Eq-Inter4}.
\end{align}
Combining \eqref{Eq-Inter1} with \eqref{Eq-Inter4} gives
\begin{equation}\label{Eq-Inter5}
	II\geq -2R^{1-\alpha}||u||_{L^{\infty}(\mathbb{H}^n)}\int_{\mathbb{H}^n}\int_{B_{2R}(x)\setminus B_R(x)}\int_{0}^1{\big|\nabla u(\gamma_{x,y}(t))\big|_g}\,dt\,dV_g(y)\,dV_g(x).
\end{equation}
We claim that for any fixed $t\in[0,1]$, there holds
\begin{equation}\label{Eq-Inter6}
\int_{\mathbb{H}^n}\int_{B_{2R}(x)\setminus B_R(x)}\big|\nabla u(\gamma_{x,y}(t))\big|_g\,dV_g(y)\,dV_g(x)=|B_{2R}\setminus B_R|_g\int_{\mathbb{H}^n}|\nabla u|_g\,dV_g.
\end{equation}
Note that \eqref{Eq-Inter6} is obviously true for $t=0$, hence we only consider $t\in(0,1]$. Indeed, the volume element in geodesic polar coordinates is 
\begin{equation}\label{Eq-Inter7}
	dV_g=\sinh^{n-1}r\,dr\,d\sigma.
\end{equation}
Then the map $g: y\mapsto z:=\gamma_{x,y}(t)$ takes $r_y=d_g(x,y)$ to $r_z:=d_g(x,z)=td_g(x,y)=tr_y$. Hence the volume element is related by
\begin{equation}\label{Eq-Inter8}
	dV_g(y)=\frac{1}{t}\left(\frac{\sinh(r_z/t)}{\sinh r_z}\right)^{n-1}dV_g(z),
\end{equation}
and 
\begin{align}
	&\int_{\mathbb{H}^n}\int_{B_{2R}(x)\setminus B_R(x)}\big|\nabla u(\gamma_{x,y}(t)\big|_g\,dV_g(y)\,dV_g(x)\notag\\
	=&\int_{\mathbb{H}^n}\int_{B_{2tR}(x)\setminus B_{tR}(x)}\frac{|\nabla u(z)|_g}{t}\left(\frac{\sinh(r_z/t)}{\sinh r_z}\right)^{n-1}dV_g(z)\,dV_g(x)\notag\\
	=&\int_{\mathbb{H}^n}|\nabla u(z)|_g\int_{\{x:z\in B_{2tR}(x)\setminus B_{tR}(x)\}}\frac{1}{t}\left(\frac{\sinh(r_z/t)}{\sinh r_z}\right)^{n-1}dV_g(x)\,dV_g(z)\notag\\
	=&\int_{\mathbb{H}^n}|\nabla u(z)|_g\int_{B_{2tR}(z)\setminus B_{tR}(z)}\frac{1}{t}\left(\frac{\sinh(r_x/t)}{\sinh r_x}\right)^{n-1}dV_g(x)\,dV_g(z),\label{Eq-Inter9}
\end{align}
where $r_x=d_g(x,z)=r_z$. Note that 
\begin{equation}\label{Eq-Inter10}
	\int_{B_{2tR}(z)\setminus B_{tR}(z)}\frac{1}{t}\left(\frac{\sinh(r_x/t)}{\sinh r_x}\right)^{n-1}dV_g(x)=\int_{B_{2R}(z)\setminus B_{R}(z)}dV_g=|B_{2R}\setminus B_R|_g,
\end{equation} 
where we have used \eqref{Eq-Inter8}. Combining \eqref{Eq-Inter9} with \eqref{Eq-Inter10} gives the claim \eqref{Eq-Inter6}, and then by \eqref{Eq-Inter5}, we conclude
\begin{align}
	II&\geq -2R^{1-\alpha}||u||_{L^{\infty}(\mathbb{H}^n)}|B_{2R}\setminus B_R|_g\int_{\mathbb{H}^n}|\nabla u|_g\,dV_g\notag\\
	&=-2n b_n R^{1-\alpha}||u||_{L^{\infty}(\mathbb{H}^n)}\int_R^{2R}\sinh^{n-1}r\,dr\int_{\mathbb{H}^n}|\nabla u|_g\,dV_g.\label{Eq-Inter11}
\end{align}
Combining \eqref{Eq-Inter1} with \eqref{Eq-Inter2} and \eqref{Eq-Inter11}, we get 
\begin{align}
\int_{\mathbb{H}^n}u^2\,dV_g&\leq \frac{(2R)^{\alpha}}{nb_n\int_R^{2R}\sinh^{n-1}r\,dr}\int_{\mathbb{H}^n}\int_{\mathbb{H}^n}\frac{u(x)u(y)}{d_g(x,y)^{\alpha}}\,dV_g(x)\,dV_g(y)+2^{\alpha+1}R||u||_{L^{\infty}}\int_{\mathbb{H}^n}|\nabla u|_g\,dV_g\notag\\
&\leq \frac{2^{\alpha}R^{\alpha-n}}{(2^n-1)b_n}\int_{\mathbb{H}^n}\int_{\mathbb{H}^n}\frac{u(x)u(y)}{d_g(x,y)^{\alpha}}\,dV_g(x)\,dV_g(y)+2^{\alpha+1}R||u||_{L^{\infty}}\int_{\mathbb{H}^n}|\nabla u|_g\,dV_g\notag\\
&:=aR^{\alpha-n}+b R,\label{Eq-Inter12}
\end{align}
where we have used $\int_R^{2R}\sinh^{n-1}r\,dr\geq \int_R^{2R}r^{n-1}\,dr=\frac{2^n-1}{n}R^n$ in the second inequality.
For the function $h(R)=aR^{\alpha-n}+bR$, we have $h'(R)=a(\alpha-n)R^{\alpha-n-1}+b$. Denote
\begin{equation}\label{Eq-Inter13}
	R_0=\left(\frac{a(n-\alpha)}{b}\right)^{\frac{1}{n+1-\alpha}}
\end{equation}
Then $h(R)$ is decreasing on $(0,R_0)$ and increasing on $(R_0,\infty)$. Hence taking $R=R_0$, where $a, b$ are defined in \eqref{Eq-Inter12} yields this lemma. 
\end{proof}
\begin{cor}\label{Cor-Inter}
	Let $F\subset\mathbb{H}^n$ be a set of finite perimeter with $|F|_g=m$, then there exists a constant $c_8$ depending only on $n$ and $\alpha$, such that
	\begin{equation}\label{Cor-Inter1}
		m\leq c_8P(F)^{\frac{n-\alpha}{n+1-\alpha}}NL_{\alpha}(F)^{\frac{1}{n+1-\alpha}}.
	\end{equation}
\end{cor}
\begin{proof}
	Since $F$ is a set of finite perimeter, then the characteristic function $\chi_F$ satisfies $\chi_F\in BV(\mathbb{H}^n)\cap L^{\infty}(\mathbb{H}^n)$. Then Corollary \ref{Cor-Inter} follows from \eqref{Lem-Inter1} with $u=\chi_F$.
\end{proof}
Then we prove that the infimum of the functional $\mathcal{E}(\cdot)$ can be controlled two-sided by the volume $m$.
\begin{thm}\label{Thm-two}
 For all $n\geq 2$, $0<\alpha<n$ and $\gamma>0$, there exist two positive constants $C_1'$ and $C_2'$, depending only on $n,\alpha$ and $\gamma$, such that
 \begin{equation}\label{Thm-Bound}
 	C_1'\max\{m^{\frac{n-1}{n}},m\}\leq \inf_{|F|_g=m}{\mathcal{E}(F)}\leq C_2'\max\{m^{\frac{n-1}{n}},m\}.
 \end{equation}
 Furthermore, for $m\geq 1$, we have that for every set of finite perimeter $F\subset\mathbb{H}^n$ satisfying $|F|_g=m$ and $\mathcal{E}(F)\leq C_3' m$ with some $C_3'>0$, then 
\begin{equation}\label{Thm-Bound1}
	C_4' m\leq\min\{P(F), NL_{\alpha}(F)\}\leq \max\{P(F), NL_{\alpha}(F)\}\leq \frac{C_3' m}{\min\{1,\gamma\}}
\end{equation}
for some $C_4'>0$ depending only on $n, \alpha, \gamma$ and $C_3'$, but not on $m$.
\end{thm}
 \begin{proof}
 	By Corollary \ref{Cor-Inter}, we have
 \begin{align}
 	m&\leq c_8P(F)^{\frac{n-\alpha}{n+1-\alpha}}NL_{\alpha}(F)^{\frac{1}{n+1-\alpha}}\notag\\
 	&\leq c_8\gamma^{-\frac{1}{n+1-\alpha}}\mathcal{E}(F).\label{Eq-two}
 \end{align}
 On the other hand,
 \begin{equation}\label{Eq-two1}
 	\mathcal{E}(F)\geq P(F)\geq n b_n^{\frac{1}{n}}|F|_g^{\frac{n-1}{n}}=n b_n^{\frac{1}{n}} m^{\frac{n-1}{n}},
 \end{equation}
 where we have used \eqref{Eq-HEin} in the second inequality. Then combining \eqref{Eq-two} with \eqref{Eq-two1} gives the lower bound in \eqref{Thm-Bound}. The upper bound in \eqref{Thm-Bound} is due to 
 \begin{equation*}
  \inf_{|F|_g=m}{\mathcal{E}(F)}\leq \limsup_{R\to\infty}\mathcal{E}({\widehat{F}_R})\leq c_5\max\{m,m^{\frac{n-1}{n}}\}
\end{equation*}
by Lemma \ref{Lem-ER}.

Now assume that $|F|_g=m$ and $\mathcal{E}(F)\leq C_3' m$.
Then 
\begin{align}
	m&\leq c_8P(F)^{\frac{n-\alpha}{n+1-\alpha}}NL_{\alpha}(F)^{\frac{1}{n+1-\alpha}}\notag\\
	&=c_8\gamma^{-\frac{1}{n+1-\alpha}}P(F)^{\frac{n-\alpha}{n+1-\alpha}}\left(\gamma NL_{\alpha}(F)\right)^{\frac{1}{n+1-\alpha}}\notag\\
	&\leq c_8 (C_3')^{\frac{1}{n+1-\alpha}}\gamma^{-\frac{1}{n+1-\alpha}}P(F)^{\frac{n-\alpha}{n+1-\alpha}}m^{\frac{1}{n+1-\alpha}},\label{Eq-two2}
\end{align}
which implies that 
\begin{equation}\label{Eq-two3}
	P(F)\geq c_8^{-\frac{n+1-\alpha}{n-\alpha}}\left(\frac{\gamma}{C_3'}\right)^{\frac{1}{n-\alpha}}m.
\end{equation}
Arguing similarily as \eqref{Eq-two2} using $P(F)\leq \mathcal{E}(F)\leq C_3' m$ gives 
\begin{equation}\label{Eq-two4}
	NL_{\alpha}(F)\geq c_8^{\alpha-n-1}(C_3')^{\alpha-n}m.
\end{equation}
Combining \eqref{Eq-two3} with \eqref{Eq-two4} gives the lower bound of \eqref{Thm-Bound1} with the constant
\begin{equation}\label{EQ-two5}
	C_4'=\min\left\{c_8^{-\frac{n+1-\alpha}{n-\alpha}}\left(\frac{\gamma}{C_3'}\right)^{\frac{1}{n-\alpha}}, c_8^{\alpha-n-1}(C_3')^{\alpha-n}\right\}.
\end{equation}
The upper bound of \eqref{Thm-Bound1} follows from the fact that
\begin{equation*}
	\max\{P(F),\gamma NL_{\alpha}(F)\}\leq\mathcal{E}(F)\leq C_3' m.
\end{equation*}
 \end{proof}
In the next lemma, we show that for a minimizer $E$ with volume $m\geq 1$, the diameter of $\bar{E}^M$ is two-sided bounded by $m$.
\begin{lem}\label{Lem-DB}
	 For all $n\geq 2$, $0<\alpha<n$ and $\gamma>0$, let $E\subset\mathbb{H}^n$ be a minimizer of the functional $\mathcal{E}(\cdot)$ with $|E|_g=m\geq 1$. Then
\begin{equation}\label{Lem-DB1}
	C_5' m^{\frac{1}{\alpha}}\leq \mathrm{diam}\, \bar{E}^M\leq C_6' m,
\end{equation} 
for two positive constants $C_5'$ and $C_6'$, depending only on $n, \alpha$ and $\gamma$. 
\end{lem}
\begin{proof}
	Denote $d:=\mathrm{diam}\,\bar{E}^M$. Since $E$ is essentially bounded and indecomposable by Theorem \ref{Thm-EI}, hence $d<\infty$ and 
	\begin{align}
		NL_{\alpha}(E)&=\int_{E}\int_{E}{\frac{1}{d_g(x,y)^{\alpha}}}\,dV_g(x)\,dV_g(y)\notag\\
		&=\int_{\bar{E}^M}\int_{\bar{E}^M}{\frac{1}{d_g(x,y)^{\alpha}}}\,dV_g(x)\,dV_g(y)\notag\\
		&\geq\frac{m^2}{d^{\alpha}}.\label{Eq-EI}
	\end{align}
	Also, by Theorem \ref{Thm-two}, 
	\begin{equation}\label{Eq-EI1}
		NL_{\alpha}(E)\leq \frac{C_2' m}{\gamma}.
	\end{equation}
	Then combining \eqref{Eq-EI} with \eqref{Eq-EI1} gives the lower bound of \eqref{Lem-DB1} with 
	\begin{equation}\label{Eq-EI2}
		C_5'=\left(\frac{\gamma}{C_2'}\right)^{\frac{1}{\alpha}}.
	\end{equation}
	We next turn to prove the upper bound of \eqref{Lem-DB1}. We know that in the upper-half space model $U^n$, the Euclidean hyperplane $\{x_{n}=t:t>0\}$ is a family of horospheres (\cite[Theorem 4.6.5]{Ratcliffe19}), which we denote as $\Pi_t$. It is easily seen that the distance between $\Pi_t$ and $\Pi_s$ satisfies $d_g(\Pi_t, \Pi_s)=|\log(\frac{s}{t})|$ for any $s, t>0$ by \eqref{Ex-metric}. Clearly, we may assume $d>6$ (otherwise $d\leq 6m$ and the upper bound follows). Without loss of generality, we may assume that there exist two points $x^{(1)}, x^{(2)}\in\bar{E}^M$, such that the $n$-th coordinate $x_n^{(1)}<1$ and $x_n^{(2)}>\mathrm{e}^{d-1}$. Let $N$ be the largest integer no larger than $\frac{d-2}{3}$. Since $E$ is indecomposable, there exists $N$ disjoint balls $B_1(x^j), j=1,\dots,N$, with $x^j\in\bar{E}^M$ such that the $n$-th coordinate of $x^j$ satisfies $\mathrm{e}^{3j-1}<x^j_n<\mathrm{e}^{3j}$. Combined with Lemma \ref{Lem-ULB}, we have
	\begin{equation*}
		m=|E|_g\geq \sum_{j=1}^N |E\cap B_1(x^j)|_g\geq c_6 N\geq \frac{d-5}{3}c_6= \frac{d-5}{3d}c_6\cdot d\geq \frac{c_6}{18}d,
	\end{equation*}
	which gives the upper bound of \eqref{Lem-DB1} with $C_6'=\max\{\frac{18}{c_6}, 6\}$. 
\end{proof}
{\noindent\bf{Proof of Theorem \ref{Thm-Nonexistence}.}} Assume that $E\subset\mathbb{H}^n$ is a minimizer of the functional $\mathcal{E}(\cdot)$ with $|E|_g=m$, we only need to prove that $m$ has an upper bound depending only on $n, \alpha$ and $\gamma$. If $\alpha<1$, then by Lemma \ref{Lem-DB}, we have an upper bound for $m$. Hence in the following, we assume that $1\leq\alpha<2$. Denote $d:=\mathrm{diam}\,\bar{E}^M<\infty$ and for any $t>0$,
\begin{align}
	\Pi_t&:=\{x=(x_1,\dots,x_n)\in U^n:x_n=t\},\label{Eq-Non0}\\
	U(t)&:=|E\cap\{x=(x_1,\dots,x_n)\in U^n: 1<x_n<t\}|_g. \label{Eq-Non}
\end{align}
Also, the function $u(x)=x_n$ defined in $U^n$ has gradient $|\nabla u|_g=x_n$. By a suitable isometry of $U^n$, we may assume that there are two points $x^{(1)}, x^{(2)}\in\bar{E}^M$ such that their $n$-th coordinates satsify $x^{(1)}_n< \mathrm{e}$ and $x^{(2)}_n>\mathrm{e}^{d-1}$ respectively, with
\begin{equation}\label{Eq-Non1}
   U(1)=0,\quad	 U(\mathrm{e}^d)=m,\quad U(\mathrm{e}^{\frac{d}{2}})\leq\frac{m}{2}.
\end{equation}
Without loss of generality, we may assume that $d>8$ (Otherwise, Lemma \ref{Lem-DB} gives an upper bound of $m$) and $m>1$. For a given $t\in(1,\mathrm{e}^{\frac{d}{2}}]$, we cut the set $E$ by the horpsphere $\Pi_t$ into two pieces $E_1$ and $E_2$, which are then moved apart to a large distance $R>0$, with the new set denoted as $E_t^R=E_{1,t}^R\cup E_{2,t}^R$.
Define
\begin{equation}\label{Eq-Non2}
	\rho(t)=\frac{1}{2}\left(P(E_t^R)-P(E)\right)=H^{n-1}(\mathring{E}^M\cap\Pi_t),
\end{equation}
where the second equality follows from the same argument as in the proof of Lemma \ref{Lem-ULB}. On the other hand,
\begin{align}
	&\gamma NL_{\alpha}(E_t^R)-\gamma NL_{\alpha}(E)\notag\\
	=&2\gamma\int_{E_{1,t}^R}\int_{E_{2,t}^R}{\frac{1}{d_g(x,y)^{\alpha}}}\,dV_g(x)\,dV_g(y)-2\gamma\int_{E_1}\int_{E_2}{\frac{1}{d_g(x,y)^{\alpha}}}\,dV_g(x)\,dV_g(y)\notag\\
	\leq&\frac{2\gamma}{R^{\alpha}}m^2-\frac{2\gamma}{d^{\alpha}} U(t)(m-U(t))\notag\\
	\leq &\frac{2\gamma}{R^{\alpha}}m^2-\frac{\gamma m}{d^{\alpha}}U(t),\label{Eq-Non3}
\end{align}
where we have used the assumption that $U(t)\leq U(\frac{d}{2})\leq \frac{m}{2}$ in the last inequality. Then combining \eqref{Eq-Non2} with \eqref{Eq-Non3} gives
\begin{equation}\label{Eq-Non4}
	\mathcal{E}(E_t^R)-\mathcal{E}(E)\leq 2\rho(t)+\frac{2\gamma}{R^{\alpha}}m^2-\frac{\gamma m}{d^{\alpha}}U(t).
\end{equation}
Since $E$ is a minimizer, hence letting $R\to\infty$ in \eqref{Eq-Non4} gives
\begin{equation}\label{Eq-Non5}
	\rho(t)\geq \frac{\gamma m}{2d^{\alpha}}U(t).
\end{equation}
On the other hand, by the coarea formula, 
\begin{equation}\label{Eq-Non6}
	U(t)=\int_{1}^t\int_{\mathring{E}^M\cap\Pi_{t'}}\frac{1}{|\nabla u|_g}\,d\mu\,dt'=\int_{1}^t\frac{\rho(t')}{t'}\,dt',
\end{equation}
where we have used \eqref{Eq-Non2} and the fact $|\nabla u|_g=x_n=t'$ on the horosphere $\Pi_{t'}$. 
Combining \eqref{Eq-Non5} with \eqref{Eq-Non6} gives
\begin{equation}\label{Eq-Non7}
	\frac{d U(t)}{dt}=\frac{\rho(t)}{t}\geq \frac{\gamma m}{2d^{\alpha}}\frac{U(t)}{t},\quad\text{a.e.} \, t\in(\mathrm{e}^{\frac{d}{4}},\mathrm{e}^{\frac{d}{2}}),
\end{equation} 
which implies that the function $V(t)$ defined  by
\begin{equation*}
	V(t)=U(t)t^{-\frac{\gamma m}{2d^{\alpha}}} 
\end{equation*}
satisfies
\begin{equation}\label{Eq-Non8}
	\frac{d}{dt}V(t)\geq 0, \quad\text{a.e.} \, t\in(\mathrm{e}^{\frac{d}{4}},\mathrm{e}^{\frac{d}{2}}),
\end{equation}
Then we have in turn that 
\begin{equation}\label{Eq-Non9}
	U(t)\geq U(\mathrm{e}^{\frac{d}{4}})\left(\frac{t}{\mathrm{e}^{\frac{d}{4}}}\right)^{\frac{\gamma m}{2d^{\alpha}}}, \quad t\in[\mathrm{e}^{\frac{d}{4}},\mathrm{e}^{\frac{d}{2}}].
\end{equation}
Since we assume that $d>8$ and $m>1$, hence we have
\begin{equation}\label{Eq-Non10}
  U(\mathrm{e}^{\frac{d}{4}})\geq 	|E\cap B_1(x^{(1)})|_g\geq c_6\min\{1,m\}=c_6
\end{equation}
by Lemma \ref{Lem-ULB}. Taking $t=\mathrm{e}^{\frac{d}{2}}$ in \eqref{Eq-Non9}, and using \eqref{Eq-Non1}, \eqref{Eq-Non10} and Lemma \ref{Lem-DB}, we conclude
\begin{equation}
	m\geq 2c_6\mathrm{e}^{\frac{\gamma}{8}md^{1-\alpha}}\geq 2c_6\mathrm{e}^{\frac{\gamma}{8}{(C_6')}^{1-\alpha}m^{2-\alpha}},
\end{equation}
which yields an upper bound for $m$ provided $\alpha<2$. This completes the proof of Theorem \ref{Thm-Nonexistence}.

\appendix
\section{The quantitive estimates in Theorem \ref{Thm-Fuglede}}
In this appendix, we prove claims \eqref{Eq-F10.2} and \eqref{Eq-F10.5} in Theorem \ref{Thm-Fuglede}. For the convenience of readers, we recall some notations. Assume that $n\geq 2$, $0<\alpha<n$ and $R_0>0$, $u\in C^1(\mathbb{S}^{n-1})$ with $||u||_{C^1(\mathbb{S}^{n-1})}\leq\frac{1}{2}$, $0<t\leq 2\varepsilon_0$ for some $\varepsilon_0\in(0,\frac{1}{2}]$. If $E_t$ is a nearly spherical set with $|E_t|_g=|B_r|_g$ for some $0<r\leq R_0$, whose boundary is given by 
	\begin{equation*}
	\partial E_t=\{(x,r(1+tu(x))):x\in\mathbb{S}^{n-1}\}.
\end{equation*}
By the previous arguments in Theorem \ref{Thm-Fuglede}, we already have
\begin{equation}\label{Eq-a0}
	NL_{\alpha}(B_r)-NL_{\alpha}(E_t)=NL_{\alpha}(B_r)-I+\frac{t^2 r^2}{2}h(t),
\end{equation}
where
\begin{align}
	NL_{\alpha}(B_r)&=\int_{\mathbb{S}^{n-1}}\,d\sigma_x\int_{\mathbb{S}^{n-1}}\,d\sigma_y\int_0^{r}\int_0^{r}{f_{|x-y|}(\rho,s)}\,d\rho\,ds,\label{Eq-a1}\\
	I&=\int_{\mathbb{S}^{n-1}}\,d\sigma_x\int_{\mathbb{S}^{n-1}}\,d\sigma_y\int_0^{r(1+tu(x))}\int_0^{r(1+tu(x))}{f_{|x-y|}(\rho,s)}\,d\rho\,ds,\label{Eq-a2}\\
	h(t)&=\int_{\mathbb{S}^{n-1}}\,d\sigma_x\int_{\mathbb{S}^{n-1}}\,d\sigma_y\int_{u(y)}^{u(x)}\int_{u(y)}^{u(x)}{f_{|x-y|}(r(1+t\rho),r(1+ts))}\,d\rho\,ds.\label{Eq-a3}
\end{align}
Here the function $f_{|x-y|}(\rho,s)$ is defined as 
\begin{equation*}
		f_{|x-y|}(\rho,s)=\frac{\sinh^{n-1}\rho\sinh^{n-1}s}{\mathrm{arccosh}^{\alpha}\left[\cosh(\rho-s)+\frac{\sinh\rho\sinh s}{2}{|x-y|}^2\right]}.
\end{equation*}
Then by change of variables, we see that
\begin{equation}\label{Eq-a4}
	I=\int_{\mathbb{S}^{n-1}}(1+tu(x))^2\,d\sigma_x\int_{\mathbb{S}^{n-1}}\,d\sigma_y\int_0^{r}\int_0^{r}{f_{|x-y|}\left((1+tu(x))\rho,(1+tu(x))s\right)}\,d\rho\,ds.
\end{equation}
\begin{lem}\label{Lem-appen--1}
	There exists a constant $\varepsilon_0\in(0,\frac{1}{2}]$ depending on $n$ and $R_0$, such that for $0<t\leq 2\varepsilon_0$,  we have
	\begin{equation}\label{Eq-a5}
	f_{|x-y|}\left((1+tu(x))\rho,(1+tu(x))s\right)\geq f_{|x-y|}(\rho,s)\left(1+tuG_{|x-y|}(\alpha,\rho,s)+t^2u^2D\right)
	\end{equation}
	for some function $G_{|x-y|}(\alpha,\rho,s)$ and a positive constant $D$ depending on $n,\alpha$ and $R_0$. Moreover, there exists a constant $D'$ depending on $n,\alpha$ and $R_0$, such that $\big|G_{|x-y|}(\alpha,\rho,s)\big|\leq D'$.
\end{lem}
\begin{proof}
	Firstly, by Taylor expansion, we have 
	\begin{align*}
		\sinh^{n-1}\left[(1+tu)\rho\right]&\geq\sinh^{n-1}\rho\left[1+(n-1)tu\rho\coth\rho\right],\\
		\sinh^{n-1}\left[(1+tu)s\right]&\geq\sinh^{n-1}s\left[1+(n-1)tu s\coth s\right]
	\end{align*}
	and hence
	\begin{equation}\label{Eq-b1}
		\sinh^{n-1}\left[(1+tu)\rho\right]\sinh^{n-1}\left[(1+tu)s\right]\geq \sinh^{n-1}\rho\sinh^{n-1}s(1+tu a_1(\rho,s)+t^2u^2a_2(\rho,s)),
	\end{equation}
	where 
	\begin{align}
		a_1(\rho,s)&=(n-1)(\rho\coth\rho+s\coth s),\label{Eq-b2}\\
		a_2(\rho,s)&=(n-1)^2\rho s\coth\rho\coth s,\label{Eq-b3}
	\end{align}
	and $a_1(\rho,s), a_2(\rho,s)$ are all uniformly bounded from above by a constant depending only on $n$ and $R_0$. Secondly, also by a direct Taylor expansion, we find
	\begin{align*}
		\cosh\left[(1+tu)(\rho-s)\right]&\leq \cosh(\rho-s)+tu(\rho-s)\sinh(\rho-s)+\frac{t^2u^2}{2}\cosh\left(\frac{3}{2}R_0\right)(\rho-s)^2,\\
		\sinh\left[(1+tu)\rho\right]&\leq \sinh\rho+tu\rho\cosh\rho+\frac{t^2u^2}{2}\rho^2\sinh\left(\frac{3}{2}\rho\right),\\
		\sinh\left[(1+tu)s\right]&\leq \sinh s+tu s\cosh s+\frac{t^2u^2}{2}s^2\sinh\left(\frac{3}{2}s\right),
	\end{align*}
	which gives
	\begin{align}
		&\cosh\left[(1+tu)(\rho-s)\right]+\frac{\sinh\left[(1+tu)\rho\right]	\sinh\left[(1+tu)s\right]}{2}|x-y|^2\notag\\
		\leq&\cosh(\rho-s)+\frac{\sinh\rho\sinh s}{2}|x-y|^2+\sum_{k=1}^4{t^k u^k}b^k_{|x-y|}(\rho,s),\label{Eq-b4}
	\end{align}
	where the functions $b^k_{|x-y|}(\rho,s)(1\leq k\leq 4)$ are defined as
	\begin{align*}
		b^1_{|x-y|}(\rho,s)&=(\rho-s)\sinh(\rho-s)+\frac{|x-y|^2}{2}(s\cosh s\sinh\rho+\rho\cosh\rho\sinh s),\\
		b^2_{|x-y|}(\rho,s)&=\frac{1}{2}\cosh\left(\frac{3}{2}R_0\right)(\rho-s)^2+\frac{|x-y|^2}{4}\left[s^2\sinh\left(\frac{3}{2}s\right)\sinh\rho+2\rho s\cosh\rho\cosh s+\rho^2\sinh\left(\frac{3}{2}\rho\right)\sinh s\right],\\
		b^3_{|x-y|}(\rho,s)&=\frac{|x-y|^2}{4}\rho s\left[s\cosh\rho\sinh\left(\frac{3}{2}s\right)+\rho\cosh s\sinh\left(\frac{3}{2}\rho\right)\right],\\
		b^4_{|x-y|}(\rho,s)&=\frac{|x-y|^2}{8}\rho^2 s^2\sinh\left(\frac{3}{2}\rho\right)\sinh\left(\frac{3}{2}s\right).
	\end{align*} 
	Denote the right hand side of \eqref{Eq-b4} as $F(t)$, then we expand $\mathrm{arccosh}(F(t))$ at the point $\mathrm{arccosh}(F(0))$. We get
	\begin{equation}\label{Eq-b5}
		\mathrm{arccosh}(F(t))=\mathrm{arccosh}(F(0))+tu\frac{b^1_{|x-y|}(\rho,s)}{\left[F^2(0)-1\right]^{\frac{1}{2}}}+\frac{F''(\tau)}{2}t^2
	\end{equation}
	for some point $\tau\in(0,t)$. Since it can be directly checked that for any $\tau\in(0,t)$,
	\begin{equation*}
		F^2(\tau)-1=O(F^2(0)-1),
	\end{equation*}
	and for all $1\leq k\leq 4$, 
	\begin{align*}
		b^k_{|x-y|}(\rho,s)&=O(F^2(0)-1),\\
		b^k_{|x-y|}(\rho,s)&=O(\mathrm{arccosh}^2(F(0))),
	\end{align*}
 hence
	\begin{align}
		&\mathrm{arccosh}\left\{\cosh\left[(1+tu)(\rho-s)\right]+\frac{\sinh\left[(1+tu)\rho\right]	\sinh\left[(1+tu)s\right]}{2}|x-y|^2\right\}\notag\\
		\leq&\mathrm{arccosh}(F(t))\notag\\
		\leq&\mathrm{arccosh}(F(0))\left[1+tu a^3_{|x-y|}(\rho,s)+t^2u^2a^4_{|x-y|}(\rho,s)\right].\label{Eq-b6}
	\end{align}
	Here $a^3_{|x-y|}(\rho,s), a^4_{|x-y|}(\rho,s)$ are two functions which are all uniformly bounded from above by a constant depending on $R_0$. Combining \eqref{Eq-b1} and \eqref{Eq-b6} gives
	\begin{equation}\label{Eq-b7}
			f_{|x-y|}\left((1+tu(x))\rho,(1+tu(x))s\right)\geq f_{|x-y|}(\rho,s)\frac{1+tu a_1(\rho,s)+t^2u^2a_2(\rho,s)}{\left(1+tu a^3_{|x-y|}(\rho,s)+t^2u^2a^4_{|x-y|}(\rho,s)\right)^{\alpha}}.
	\end{equation}
	Then the conclusion of this lemma follows from Taylor expansion of the right hand side of \eqref{Eq-b7} once more.
\end{proof}
\begin{lem}\label{Lem-appen_0}
	Under the choice of $\varepsilon_0$ in Lemma \ref{Lem-appen--1}, there exists a constant $L'$ depending on $n,\alpha$ and $R_0$, such that 
	\begin{equation}
		NL_{\alpha}(B_r)-I\leq t^2L'\sinh^{2n-\alpha} r||u||^2_{L^2(\mathbb{S}^{n-1})}.
	\end{equation}
\end{lem}
\begin{proof}
	In the absence of ambiguity, the function $G_{|x-y|}(\alpha,\rho,s)$ is abbreviated as $G$ in the following of this proof. By Lemma \ref{Lem-appen--1}, we see that 
	\begin{align}
		NL_{\alpha}(B_r)-I\leq& -t\int_{\mathbb{S}^{n-1}}u\,d\sigma_x\int_{\mathbb{S}^{n-1}}\,d\sigma_y\int_0^{r}\int_0^{r}(G+2)f_{|x-y|}(\rho,s)\,d\rho\,ds\notag\\
		&-t^2\int_{\mathbb{S}^{n-1}}u^2\,d\sigma_x\int_{\mathbb{S}^{n-1}}\,d\sigma_y\int_0^{r}\int_0^{r}(D+2G+1)f_{|x-y|}(\rho,s)\,d\rho\,ds\notag\\
		&-t^3\int_{\mathbb{S}^{n-1}}u^3\,d\sigma_x\int_{\mathbb{S}^{n-1}}\,d\sigma_y\int_0^{r}\int_0^{r}(2D+G)f_{|x-y|}(\rho,s)\,d\rho\,ds\notag\\
		&-Dt^4\int_{\mathbb{S}^{n-1}}u^4\,d\sigma_x\int_{\mathbb{S}^{n-1}}\,d\sigma_y\int_0^{r}\int_0^{r}f_{|x-y|}(\rho,s)\,d\rho\,ds\label{Eq-a6}
	\end{align}
	Note that we have the following two observations. Firstly, since $|E_t|_g=|B_r|_g$, we have (cf. \cite[Equation (4.5)]{Bogelein15})
	\begin{equation}\label{Eq-a7}
		\left|t\int_{\mathbb{S}^{n-1}}u\,d\sigma_x\right|\leq D_1t^2||u||^2_{L^2(\mathbb{S}^{n-1]})},
	\end{equation}
	for some constant $D_1$ depending on $n$ and $R_0$. Secondly, for any fixed $x\in\mathbb{S}^{n-1}$, the integral $\int_{\mathbb{S}^{n-1}}\,d\sigma_y\int_0^{r}\int_0^{r}f_{|x-y|}(\rho,s)\,d\rho\,ds$ and $\int_{\mathbb{S}^{n-1}}\,d\sigma_y\int_0^{r}\int_0^{r}Gf_{|x-y|}(\rho,s)\,d\rho\,ds$ are only dependent on $|x|$, hence we have
	\begin{equation}\label{Eq-a8}
		\int_{\mathbb{S}^{n-1}}\,d\sigma_y\int_0^{r}\int_0^{r}f_{|x-y|}(\rho,s)\,d\rho\,ds=\frac{NL_{\alpha}(B_r)}{\omega_{n-1}}\leq D_2\sinh^{2n-\alpha}r
	\end{equation}
	and 
	\begin{equation}\label{Eq-a9}
		\int_{\mathbb{S}^{n-1}}\,d\sigma_y\int_0^{r}\int_0^{r}Gf_{|x-y|}(\rho,s)\,d\rho\,ds\leq D'D_2\sinh^{2n-\alpha} r.
	\end{equation}
	Here $D_2$ is a constant depending on $n,\alpha$ and $R_0$ due to the inequality \eqref{In-b1}. Hence combining \eqref{Eq-a6}-\eqref{Eq-a9} with the fact that $|tu|\leq\varepsilon_0\leq\frac{1}{2}$ gives the lemma.
\end{proof}
\begin{lem}\label{Lem-appen}
	Assume that $x$ and $y$ are two arbitrary points in $\mathbb{S}^{n-1}$, the values of $\rho, s$ are taken between $u(x)$ and $u(y)$ and $0<\varepsilon_0\leq\frac{1}{2}$. Then for any $\tau\in(0,2\varepsilon_0)$ and $r\in(0,R_0]$, we have 
	\begin{equation}
			\left|\frac{\partial f_{|x-y|}(r(1+\tau\rho),r(1+\tau s))}{\partial\tau}\right|\leq L''(n,\alpha,R_0)f_{|x-y|}(r,r)
	\end{equation}
for some positive constant $L''$ depending on $n,\alpha$ and $R_0$.
\end{lem}
\begin{proof}
	Define $A=r(1+\tau\rho)$, $B=r(1+\tau s)$ and denote $\theta:=|x-y|$ for simplicity, then
	\begin{align}
		f_{\theta}(A,B)&=\sinh^{n-1}A\sinh^{n-1}B\,\text{arccosh}^{-\alpha}\left[\cosh(A-B)+\frac{\sinh A\sinh B}{2}\theta^2\right],\\
		\frac{\partial f_{\theta}(A,B)}{\partial\tau}&=\frac{\partial f_\theta}{\partial A}\frac{\partial A}{\partial\tau}+\frac{\partial f_\theta}{\partial B}\frac{\partial B}{\partial\tau}=r\left[\rho\frac{\partial f_\theta}{\partial A}+s\frac{\partial f_\theta}{\partial B}\right],
	\end{align}
	where 
	\begin{align}
		\frac{\partial f_{\theta}}{\partial A}=&(n-1)\sinh^{n-2}A\sinh^{n-1}B\cosh A\,\text{arccosh}^{-\alpha}\left[\cosh(A-B)+\frac{\sinh A\sinh B}{2}\theta^2\right]\notag\\
		&-\alpha\sinh^{n-1}A\sinh^{n-1}B\,\text{arccosh}^{-\alpha-1}\left[\cosh(A-B)+\frac{\sinh A\sinh B}{2}\theta^2\right]\notag\\
		&\times\left\{\left[\cosh(A-B)+\frac{\sinh A\sinh B}{2}\theta^2\right]^2-1\right\}^{-\frac{1}{2}}\left(\sinh(A-B)+\frac{\cosh A\sinh B}{2}\theta^2\right)\notag\\
		=:&\text{I}_A-\text{II}_A,
	\end{align}
	and by symmetry
	\begin{align}
		\frac{\partial f_{\theta}}{\partial B}=&(n-1)\sinh^{n-2}B\sinh^{n-1}A\cosh B\,\text{arccosh}^{-\alpha}\left[\cosh(A-B)+\frac{\sinh A\sinh B}{2}\theta^2\right]\notag\\
		&-\alpha\sinh^{n-1}A\sinh^{n-1}B\,\text{arccosh}^{-\alpha-1}\left[\cosh(A-B)+\frac{\sinh A\sinh B}{2}\theta^2\right]\notag\\
		&\times\left\{\left[\cosh(A-B)+\frac{\sinh A\sinh B}{2}\theta^2\right]^2-1\right\}^{-\frac{1}{2}}\left(\sinh(B-A)+\frac{\cosh B\sinh A}{2}\theta^2\right)\notag\\
		=:&\text{I}_B-\text{II}_B.
	\end{align}
	Hence we have
	\begin{align}
		\left|\frac{\partial f_{\theta}}{\partial\tau}\right|&=r\left|(\rho\text{I}_A+s\text{I}_B)-(\rho\text{II}_A+s\text{II}_B)\right|\notag\\
		&\leq\frac{r}{2}\left(\text{I}_A+\text{I}_B\right)+r|\rho\text{II}_A+s\text{II}_B|.\label{Eq-App1}
	\end{align}
In the following, we prove that $r\text{I}_A$ (then by symmety, $r\text{I}_B$) and $r|\rho\text{II}_A+s\text{II}_B|$ can be controlled by $f_{\theta}(r,r)$ respectively. Firstly, it is obvious that there exists two constant $K_1$, $K_2$ depending on $R_0$, such that $r\leq\sinh r\leq K_1 r$ and $\sinh r\leq \sinh \frac{3}{2}r\leq K_2\sinh r$ for $0\leq r\leq R_0$. Hence we know that $\sinh A\leq \sinh\frac{3}{2}r\leq K_2\sinh r$ and also $\sinh B\leq K_2\sinh r$. On the other hand, we have
\begin{align}
	\mathrm{arccosh}\left[\cosh(A-B)+\frac{\sinh A\sinh B}{2}\theta^2\right]&\geq\mathrm{arccosh}\left[1+\frac{\sinh^2(\frac{r}{2})}{2}\theta^2\right]\notag\\
	&\geq \mathrm{arccosh}\left(1+\frac{r^2\theta^2}{8}\right)\notag\\
	&\geq K_3 \mathrm{arccosh}\left(1+\frac{K_1^2}{2}r^2\theta^2\right)\notag\\
	&\geq K_3\mathrm{arccosh}\left(1+\frac{\sinh^2 r}{2}\theta^2\right)\label{Eq-App2}
\end{align}
for some constant $K_3$ depending only on $R_0$. The third inequality is due to $0\leq \xi:=\theta r\leq 2R_0$ and 
\begin{equation*}
	\lim_{\xi\to 0}\frac{\mathrm{arccosh}\left(1+\frac{\xi^2}{8}\right)}{\mathrm{arccosh}\left(1+\frac{K_1^2}{2}\xi^2\right)}=\frac{1}{2K_1}.
\end{equation*}
From the above argument, we arrive at
\begin{equation}\label{Eq-App2.5}
	\frac{r}{2}\left(\text{I}_A+\text{I}_B\right)\leq (n-1)K_2^{2n-3}K_3^{-\alpha}\cosh\left(\frac{3}{2}R_0\right)f_{\theta}(r,r).
\end{equation}
For the term $|\rho\text{II}_A+s\text{II}_B|$, we do further estimates: Note that $0\leq \sinh r\cdot\theta\leq 2\sinh R_0$, then there exists a constant $K_4>1$, depending only on $R_0$, such that 
\begin{equation}\label{Eq-App3}
	\mathrm{arccosh}\left(1+\frac{\sinh^2 r}{2}\theta^2\right)\geq \frac{\sinh r\cdot\theta}{K_4}\geq\frac{r\theta}{K_4}.
\end{equation}
Also, we have 
\begin{align}
	&\left[\cosh(A-B)+\frac{\sinh A\sinh B}{2}\theta^2\right]^2-1\notag\\
	\geq&\left(1+\frac{\sinh A\sinh B}{2}\theta^2\right)^2-1\notag\\
	\geq&\sinh A\sinh B\cdot\theta^2\notag\\
	\geq&\sinh^2\left(\frac{r}{2}\right)\cdot\theta^2\geq\frac{r^2\theta^2}{4}.\label{Eq-App4}
\end{align}
Then combining \eqref{Eq-App2}, \eqref{Eq-App3} and \eqref{Eq-App4}, we have
\begin{equation}\label{Eq-App5}
  r|\rho\text{II}_A+s\text{II}_B|\leq 2\alpha K_1 K_2^{2n-2} K_3^{-\alpha-1}K_4\left[\frac{\left(u(x)-u(y)\right)^2}{|x-y|^2}+\frac{K_2}{2}\cosh\left(\frac{3}{2}R_0\right)\right]f_{\theta}(r,r).
\end{equation}
Since it can be easily checked that the geodesic distance $d_{\mathbb{S}^{n-1}}(x,y)$ between $x$ and $y$ in $\mathbb{S}^{n-1}$ satisfies $|x-y|\leq d_{\mathbb{S}^{n-1}}(x,y)\leq \frac{\pi}{2}|x-y|$
, and also observe that $||u||_{C^1(\mathbb{S}^{n-1})}\leq\frac{1}{2}$, then there exists a constant $K_5$, depending only on the upper bound of the $C^1$ norm of $u$, such that 
\begin{equation}\label{Eq-App6}
	\frac{|u(x)-u(y)|}{|x-y|}\leq K_5.
\end{equation}
Combining \eqref{Eq-App1}, \eqref{Eq-App2.5}, \eqref{Eq-App5} and \eqref{Eq-App6} yields the assertion of this lemma.
\end{proof}
\noindent\textbf{Data Availability Statement.} Data sharing not applicable to this article as no datasets were generated or analysed during the current study.

\noindent\textbf{Conflict of interest.} On behalf of all authors, the corresponding author states that there is no conflict of
interest.

\end{document}